\documentclass[preprint,11pt,authoryear]{elsarticle}

\usepackage[utf8]{inputenc}
\usepackage{geometry}
\usepackage[utf8]{inputenc}
\usepackage{graphicx} 
\usepackage{bussproofs}
\usepackage{amsmath,amssymb,amsfonts}
\usepackage{mathrsfs,latexsym,amsthm,enumerate}
\usepackage{tikz}
\usepackage{tikz-cd}
\usetikzlibrary{cd}
\usepackage{natbib}
\usepackage{amscd}
\usepackage[all]{xy}
\usepackage[utf8]{inputenc}
\usepackage{amsmath,amssymb,amsfonts}
\usepackage{mathrsfs,latexsym,amsthm,enumerate}
\usepackage{breqn}
\usepackage{fancyhdr}
\usepackage{sectsty}
\usepackage{stmaryrd}
\usepackage{url}
\usepackage{bbm}
\usepackage{subcaption}
\usepackage{float}
\catcode`@=11
\def\caseswithdelim#1#2{\left#1\,\vcenter{\normalbaselines\m@th
  \ialign{\strut$##\hfil$&\quad##\hfil\crcr#2\crcr}}\right.}
\catcode`@=12
\def\bcases#1{\caseswithdelim[{#1}}

\allsectionsfont{\sffamily\mdseries\upshape}

\catcode`@=11
\def\caseswithdelim#1#2{\left#1\,\vcenter{\normalbaselines\m@th
  \ialign{\strut$##\hfil$&\quad##\hfil\crcr#2\crcr}}\right.}
\catcode`@=12
\def\bcases#1{\caseswithdelim[{#1}}

\allsectionsfont{\sffamily\mdseries\upshape}

\theoremstyle{plain}
\newtheorem{lemma}{Lemma}[section]
\newtheorem{proposition}[lemma]{Proposition}
\newtheorem{corollary}[lemma]{Corollary}

\newtheorem{theorem}[lemma]{Theorem}
\newtheorem{observation}[lemma]{Observation}

\theoremstyle{definition}
\newtheorem{remark}[lemma]{Remark}

\newtheorem{example}[lemma]{Example}

\newcommand{\lra}{\leftrightarrows}
\newcommand{\ra}{\rightarrow}
\newcommand{\epi}{\twoheadrightarrow}
\newcommand{\inclu}{\hookrightarrow}
\newcommand{\da}{{\downarrow}}
\newcommand{\up}{{\uparrow}}

\newcommand{\tc}{\textit}
\newcommand{\mb}[1]{\mbox{#1}}
\newcommand{\ca}{\mathcal}
\newcommand{\mf}{\mathsf}
\newcommand{\mc}[1]{\mathit{#1}}
\newcommand{\mi}{\mathit}

\newcommand{\se}{\subseteq}
\newcommand{\sm}{\setminus}

\newcommand{\we}{\wedge}
\newcommand{\ve}{\vee}
\newcommand{\bwe}{\bigwedge}
\newcommand{\bve}{\bigvee}

\newcommand{\bca}{\bigcap}
\newcommand{\bcu}{\bigcup}
\newcommand{\bd}[1]{\mathbf{#1}}

\newcommand{\Fi}{\mathsf{Filt}}
\newcommand{\Fip}{\mathsf{Filt}(L^+)}
\newcommand{\Fim}{\mathsf{Filt}(L^-)}

\newcommand{\Id}{\mathsf{Idl}}

\newcommand{\less}[2]{\nabla (#1)\cap\Delta (#2)}
\newcommand{\na}{\nabla}
\newcommand{\del}{\Delta}
\newcommand{\spa}{\mathsf{sp}}

\newcommand{\bisp}{\mathsf{bisp}}
\newcommand{\bisob}{\mathsf{bisob}}

\newcommand{\va}[3]{\varphi_{\ca{#1}}^{#2}(#3^{#2})}

\newcommand{\con}{\mathsf{con}}
\newcommand{\tot}{\mathsf{tot}}
\newcommand{\id}{\mathsf{id}}

\newcommand{\fpa}[2]{(#1^+(#2^+),#1^-(#2^-))}

\newcommand{\pa}[1] {(#1^+,#1^-)}

\newcommand{\Om}{\Omega}
\newcommand{\dpt}{\mf{dpt}}
\newcommand{\bpt}{\mf{bpt}}
\newcommand{\dOm}{\mf{d}\Om}
\newcommand{\bOm}{\mf{b}\Om}
\newcommand{\bomf}{\mf{b}\Om_{\fin}}
\newcommand{\pt}{\mathsf{pt}}
\newcommand{\Lq}[1]{\ca{L}/#1}

\newcommand{\SLop}{\mathsf{S}(L)^{op}}

\newcommand{\Ad}{\mf{A}(\ca{L})}

\newcommand{\fin}{\mf{fin}}

\newcommand{\Sfop}{\mf{S}_{\fin}(L)^{op}}

\newcommand{\sy}[1]{{{\ulcorner}} #1 {{\urcorner}}}
\newcommand{\syp}[1]{{\ulcorner} #1^+{\urcorner}}
\newcommand{\sym}[1]{{\ulcorner} #1^-{\urcorner}}
\newcommand{\syw}[2]{\syp{#1}\we \sym{#2}}
\newcommand{\syv}[2]{\syp{#1}\ve \sym{#2}}

\newcommand{\cpu}[1]{#1^+\oplus #1^-}
\newcommand{\cp}[2]{#1\oplus #2}
\newcommand{\ltl}{L^+\times L^-}

\title{The category of finitary biframes as the category of pointfree bispaces}
\author{Anna Laura Suarez\fnref{address}}
\address{CMUC, Department of Mathematics, University of Coimbra, 3001-454, Coimbra, Portugal}
\address{School of Computer Science, University of Birmingham, B15 2TT, Birmingham, UK}
\ead{axs1431@cs.bham.ac.uk}

\begin{document}

\begin{frontmatter}

\begin{abstract}
The theory of finitary biframes as order-theoretical duals of bitopological spaces is explored. The category of finitary biframes is a coreflective subcategory of that of biframes. Some of the advantages of adopting finitary biframes as a pointfree notion of bispaces are studied. In particular, it is shown that for every finitary biframe there is a biframe which plays a role analogue to that of the assembly in the theory of frames: for every finitary biframe $\ca{L}$ there is a finitary biframe $\mf{A}(\ca{L})$ with a universal property analogous to that of the assembly of a frame; and such that its main component is isomorphic to the ordered collection of finitary quotients of $\ca{L}$ (i.e. its pointfree bispaces). Furthermore, in the finitary biframe duality the bispace associated with $\mf{A}(\ca{L})$ is a natural bitopological analogue of the Skula space of the bispace associated with $\ca{L}$. 
The finitary biframe duality gives us a notion of bisobriety which is weaker than pairwise Hausdorffness, incomparable with the pairwise $T_1$ axiom, and stronger than the pairwise $T_0$ axiom.
The notion of pairwise $T_D$ bispaces is introduced, as a natural point-set generalization of the classical $T_D$ axiom. It is shown that in the finitary biframe duality this axiom plays a role analogous to that of the classical $T_D$ axiom for the frame duality. 
\end{abstract}


\end{frontmatter}

\tableofcontents

\section{Introduction}
\subsection{Bispaces}
Bitopological spaces, or \tc{bispaces}, were introduced in \cite{kelly63}. The motivation behind their introduction was the study of \tc{quasi-metric spaces}, spaces which satisfy all the same defining axioms as metric spaces except symmetry of the distance. Quasi-metric spaces are studied in detail in \cite{quasimetric31}. In \cite{quasimetricsource19}, the authors lists papers from several different mathematical fields which all use nonsymmetric distances: see \cite{OnAsymmetricDistances, quasimetric1, quasimetric2, quasimetric3, quasimetric4}. Special quasi-metric spaces that have received attention recently are \tc{quasi-Polish spaces} (see \cite{polish13}, \cite{computabilitypolish19}). Bitopological spaces generalise quasi-metric spaces in the same way in which topological spaces generalize metric ones. A \tc{bitopological space} is a triple $(X,\tau^+,\tau^-)$ in which $X$ is a set and $\tau^+$ and $\tau^-$ are both topologies on that set. A morphism between bispaces is a function between their underlying sets such that it is \tc{bicontinuous}: continuous, that is, with respect to both topologies. The category $\bd{BiTop}$ has bitopological spaces as objects and bicontinuous maps as morphisms. There is a functor $\mf{patch}:\bd{BiTop}\ra \bd{Top}$ assigning to each bitopological space $(X,\tau^+,\tau^-)$ its \tc{patch}, which is the topological space $(X,\tau^+\ve \tau^-)$, in which the join $\tau^+\ve \tau^-$ is computed in the lattice of all possible topologies on $X$, and so it is the topology generated by $\tau^+\cup \tau^-$.

\subsection{Pointfree spaces}
With his famous representation theorem, Stone proved in \cite{stone36} a result which, expressed in contemporary mathematical language, establishes a dual equivalence of categories between Boolean algebras and certain topological spaces. Stone's Representation Theorem was highly influential, and started a tradition of studying topological spaces in terms of the order theoretical properties of their lattice of open sets. After Stone, the task of the algebraic, order-theoretical study of topological spaces was undertaken by McKinsey and Tarski (see \cite{mckinsey44}), Ehresmann (see \cite{Ehresmann1958}), Papert (see \cite{papert57}), Isbell (see \cite{isbell72}), Banaschewski (see \cite{banaschewski73}), Simmons (see \cite{simmons78}), Pultr (see \cite{Pultr1984}, \cite{Pultr1984i}) and Picado (see \cite{picadopultr2011frames}). Ehresmann and B\'{e}nabou were the first to consider frames as abstract representations of topological spaces. A frame is a bounded complete lattice $L$ in which we have the following distributivity law for every $x,y_i\in L$.
\[
(\bve_i y_i)\we x=\bve_i (y_i\we x).
\]
This distributivity law captures abstractly the fact that in topological spaces arbitrary unions of opens are well-behaved, but only finite intersections are. Since Ehresmann and B\'{e}nabou put this view forward, frames are considered to be the standard order theoretical representation of topological spaces. The category of frames is called $\bd{Frm}$, and we have an adjunction $\Om:\bd{Top}\lra \bd{Frm}^{op}:\pt$ with $\Om\vdash \pt$ in which $\Om$ is a functor assigning to each space its ordered collection of open sets. It is by virtue of this adjunction that we consider frames to be pointfree versions of spaces. We refer the reader to \cite{picado2003} and to \cite{johnstone83} for more on the motivations behind pointfree topology. We only mention here the fact that many results that require the Axiom of Choice in point-set topology do not require it in pointfree topology. Among these we have Tychonoff's Theorem (see \cite{Johnstone1981}), the construction of the Stone-\v{C}ech compactification of a space (see \cite{banaschewski03}), and the construction of the Samuel compactification (see \cite{banaschewski90}).

\subsection{The assembly in the theory of pointfree spaces}

It is known that for every frame $L$ there is a frame $\mf{A}(L)$, called the \tc{assembly} of $L$, which is isomorphic to the ordered collection of all pointfree subspaces of $L$. This plays a role of great importance in pointfree topology. The assembly $\mf{A}(L)$ of a frame $L$ can be seen as a ``pointfree space of all pointfree subspaces of $L$". One aspect in which we see a crucial difference between point-set and pointfree topology is the notion of subspace: think for instance of Isbell's Density Theorem, which states that in pointfree topology every space has a smallest dense subspace (see \cite{johnstone83},\cite{isbell72}). It is not surprising, then, that the assembly has received much attention in the community and that sublocales have been often studied in terms of its structure (see, for instance, \cite{Wilson1994TheAT}, \cite{isbell72}, \cite{niefield87} and more recently \cite{picado19}, \cite{clementino18}, \cite{Avila19},\cite{avila2019frame}). In the following theorem we gather three important properties of the assembly of a frame. The first was proved in \cite{Joyal1984}, and the second and the third can be found in \cite{picadopultr2011frames}. We recall that for a space $X$ the \tc{Skula space} of $X$ is the space $Sk(X)$ such that its underlying sets is the same as that of $X$, and such that its opens are the topology generated by the opens and the closed sets of $X$.

\begin{theorem}
For a frame $L$ we have the following facts.
\begin{itemize}
    \item For every frame $L$ there is a frame injection $\na:L\ra \mf{A}(L)$ such that whenever $f:L\ra M$ is a frame map such that $f(x)$ is complemented in $M$ for every $x\in L$ there is a unique frame map $\Tilde{f}:\mf{A}(L)\ra M$ such that $\Tilde{f}\circ \na=f$.
    \item The frame $\mf{A}(L)$ is anti-isomorphic to the coframe $\mf{S}(L)$ of sublocales of $L$.
    \item There is a homeomorphism $\alpha_L:\pt(\mf{A}(L))\cong Sk(\pt(L))$, and the bijection $\alpha_L:|\pt(\mf{A}(L))|\cong |\pt(L)|$ is such that the images of closed sets of $\pt(\mf{A}(L))$ are exactly the underlying sets of the sober subspaces of $\pt(L)$.
\end{itemize}
\end{theorem}

\subsection{Pointfree bispaces}

Stone's famous representation theorem (see \cite{stone36}, but also \cite{Stone38}) was only the first in a series of results connecting lattice theory and topology. For example, Priestley proved (see \cite{priestley70}, \cite{priestley72}) that the category of distributive lattices is dually equivalent to a category whose objects are Stone spaces equipped with an order relation on the set of points. 
One usually calls \tc{Stone-type dualities} those results establishing dual equivalences between categories of lattices and categories of spaces. Stone duality continues to receive a great deal of attention from both mathematicians and computer scientists. Consider, for instance: \cite{ernestone}, an order theoretical approach to the duality; \cite{caramello2011topostheoretic}, a topos theoretical one; \cite{bezhanishvili_holliday_2020}, a generalization of the duality which does not need the Axiom of Choice; \cite{gehrke16}, an application of the duality to automata and regular languages. Priestley spaces are Stone spaces equipped with an order on their sets of points. Sometimes Priestley duality is rephrased in terms of a duality between the category of \tc{coherent spaces} and that of \tc{coherent frames}. The category of coherent spaces is not a full subcategory of that of topological spaces: morphisms between coherent spaces are required to be such that preimages of compact open sets are still compact and open. 

In \cite{bezhanishvili10} a new, bitopological view of Priestley duality is put forward. There, it is explained that coherent spaces are nothing but ``halves" of pairwise Stone spaces. We recall that a Stone space is a topological space which is compact, Hausdorff, and zero dimensional. Using very natural bitopological generalizations of these three properties one can define \tc{pairwise Stone} bispaces. The category of pairwise Stone bispaces is equivalent to that of coherent spaces, and the equivalence is given by a functor $\pi^+:\bd{BiStone}\ra \bd{CohTop}$ which acts on objects as $(X,\Om^+(X),\Om^-(X)\mapsto (X,\Om^+(X))$, ``forgetting" about one topology. The key is that the category $\bd{BiStone}$ is a full subcategory of $\bd{BiTop}$: the need for an unnatural restriction on morphisms is eliminated: the condition that maps are well-behaved with respect to compactness translates simply to bicontinuity of the map, eliminating the need for extra constraints in order to get our duality. Not only do coherent spaces miss half of a pairwise Stone bispace, but them missing this half also means that in their category we need an extra constraint on morphisms.

It is difficult, then, to argue that one can have a full understanding of Stone-type dualities without understanding what a pointfree bispace is. If coherent spaces are missing half of the picture because they forget half a bispace, then so are coherent frames. One needs to generalize frames in the same way as we have generalized spaces with bispaces. In the literature, there are two alternative translations of the category of bispaces into pointfree language. The first of these translations in chronological order is the category of \tc{biframes}. We refer the reader to \cite{banaschewski83} and to \cite{schauerte92} for the theory of biframes. The second is the category of \tc{d-frames}. For this, we refer the reader to \cite{jung06} and to \cite{jakl2018}.

\subsection{The aim of this article}
In this article we explore a new notion of pointfree bispace is, in order to have a theory of pointfree bisubspaces which reflects the crucial aspects of the theory of pointfree subspaces as much as possible. Below, we list some discrepancies between the theories of biframes and d-frames and that of frames. The recurring theme is the following: it is a key feature of pointfree topology that the ordered collection of subspaces of a pointfree space is itself a pointfree space (with interesting connections to the original space, too), and when we move to both biframes and d-frames we lose this.
\begin{itemize}
    \item The object $\Ad$ with the universal property analogous to that of the assembly does not represent the collection of all quotients of $\ca{L}$ in any meaningful way, both in the category of d-frames and in that of biframes. 
   \item The lattice of all quotients of a d-frame is not distributive in general, which makes it difficult to work with it. The lattice of all quotients of a biframe collapses to the lattice of all quotients of its main component. This means that the notion is not bitopological: two distinct biframes with the same main component have the same lattice of quotients.
   \item One might argue that for a biframe $\ca{L}$ the biframe $(\mf{Cl}(L),\mf{Fitt}(L),\mf{A}(L))$ -- the so-called \tc{Skula biframe} of $\ca{L}$ -- is a suitable candidate to play the role of a ``pointfree bispace of all pointfree bisubspaces" of $\ca{L}$: the frame $\mf{A}(L)$ is anti-isomorphic to the ordered collection of biframe quotients of $\ca{L}$, and as this comes equipped with a natural bitopological structure as it is generated by the frame of the closed congruences together with that of the fitted one. However, the bitopological structure of $\ca{L}$ itself is completely lost here. If we took $(\mf{Cl}(L),\mf{Fitt}(L),\mf{A}(L))$ as playing this role, we would again obtain that any two biframes with the same patch have the same pointfree bisubspace of bisubspaces.  
   
   \item On the spatial side, the spatial version of the assembly is not as well-behaved as the analogous monotopological construction, either. As we have seen if we take the spectrum of the assembly $\pt(\mf{A}(L))$ for a frame $L$ we obtain a perfect encoding of spatial sublocales of $L$. In particular, the frame of opens of $\pt(\mf{A}(L))$ is anti-isomorphic to the lattice $\spa[\mf{S}(L)]$ of spatial sublocales of $L$. The analogue of this is false for both d-frames and biframes. In particular, for a d-frame $\ca{L}$, we have that the spatial d-quotients of a d-frame do not form a distributive lattice in general.
    
    \item For spaces, Hausdorffness implies sobriety. We do not know whether pairwise Hausdorffness implies d-sobriety of a bispace. In \cite{jung06} it is only shown that a strengthened version of pairwise Hausdorffness implies d-sobriety.
    
    \item The biframe version of sobriety is equivalent to sobriety of the patch of the relevant bispace. Therefore, the notion is not bitopological. For two distinct bispaces with the same patch, one is ``bisober" if and only if the other one is: the bitopological structure of the bispace is irrelevant.

\end{itemize}
In this article it will be shown that it is sufficient to modify the duality of bispaces and biframes to involve only \tc{finitary} biframes to mend all these discrepancies. This article is based on Chapter 9 of \cite{suarez20}.

\section{Preliminaries}
\subsection{Biframes and their duality}\label{bifrmdual0}

We now look at the most literal translation of the notion of bispace into pointfree language. We will introduce an adjunction similar to the spatial-sober one for frames and spaces, relating the category of bitopological spaces with a possible notion of a pointfree version of it. The theory that we describe here is introduced in \cite{banaschewski83}, and further developed in \cite{schauerte92}. A \tc{biframe} is a triple $(L^+,L^-,L)$ where $L^+,L^-,L$ are frames, together with frame injections $e_L^+:L^+\ra L$ and $e_L^-:L^-\ra L$ such that all elements of $L$ are of the form $\bve_i e_L^+(x^+_i)\we e_L^-(x^-_i)$ for families $x^+_i\in L^+$ and $x^-_i\in L^-$. In general, we denote a biframe just as a triple $(L^+,L^-,L)$, leaving it implicit that we call the two subframe inclusions $e^+_L$ and $e^-_L$ as above. Additionally, when we a biframe with a triple $(L^+,L^-,L)$ in which $L^+,L^-\se L$ we will leave it implicit that we take the frame injections are the two subset inclusions. Additionally, when we have a biframe with triple $(L^+,L^-,L)$ where $L$ is a free construction with generators from $L^+$ and $L^-$ we will leave it implicit that we will regard the syntactic map $\sy{-}:L^+\ra L$, and the analogous one on the negative component, to be the two frame injections. To simplify further, we will denote a biframe $(L^+,L^-,L)$ as $\ca{L}$, a biframe $(M^+,M^-,M)$ as $\ca{M}$, and so on. For a biframe $\ca{L}$, we call $L$ the \tc{main component} of a biframe.
We note that for any biframe $L$ there is always a congruence $C$ on $\cpu{L}$ such that we have an isomorphism $i:(\cpu{L})/C\cong L$ which is well-behaved with respect to the subframe inclusions, in the sense that the following diagram and the analogous one for the negative component commute.
\[
\begin{tikzcd}[row sep=large, column sep=large]
L^+
\arrow[d,"q_c\circ \sy{-}",swap]
\arrow{dr}{e^+_L}
\\
(\cpu{L})/C
\arrow{r}{i}
& L
\end{tikzcd}
\]
Here, $q_C:\cpu{L}\ra (\cpu{L})/C$ is the quotient map, and $\sy{-}:L^+\ra \cpu{L}$ the canonical injection of $L^+$ into the coproduct. 
For a biframe $\ca{L}$, in particular, the quotient of $\cpu{L}$ corresponding to $L$ is the quotient by the relation 
\[
R_L:=\{(\bve_i\syw{x_i}{x_i},\bve_j\syw{y_j}{y_j}):\bve_i e^+_L(x^+_i)\we e^-_L(x^-_i)\leq \bve_j e^+_L(y^+_j)\we e^-_L(y^-_j)\mb{ in }L\}.
\]
This is because all of the order structure of $L$ is determined by inequalities of the form $\bve_i e^+_L(x^+_i)\we e^-_L(x^-_i)\leq \bve_j e^+_L(y^+_j)\we e^-_L(y^-_j)$, and so the quotient $(\cpu{L})/R_L$ captures all equalities that hold in $L$, and no more than them, since all of the extra equalities are consequences of the other ones and the frame laws. 
A morphism between biframes $\ca{L}$ and $\ca{M}$ is a pair of frame maps $\pa{f}$ with $f^+:L^+\ra M^+$ and $f^-:L^-\ra M^-$ such that there is a frame map $f:L\ra M$ making the following diagram and the analogous one on the negative component both commute.
\[
\begin{tikzcd}[row sep=large, column sep=large]
L^+
\arrow{d}{e^+_L}
\arrow{r}{f^+}
& M^+
\arrow{d}{e^+_M}
\\
L
\arrow{r}{f}
& M
\end{tikzcd}
\]
In the particular case where we have $L^+,L^-\se L$ and $M^+,M^-\se M$, thus, for a pair of suitable frame maps $\pa{f}$ for it to count as a morphism $\pa{f}:\ca{L}\ra \ca{M}$ we need there to be a frame map $f:L\ra M$ that restricts to $L^+$ as $f^+$ and as $f^-$ to $M^-$. Sometimes, as an abbreviation, we will refer as $f$ to a biframe map $\pa{f}:\ca{L}\ra \ca{M}$. It will always be clear from the context whether we are speaking about the biframe map itself or the frame map $f:L\ra M$ witness of $\pa{f}$ being a biframe map. Let us see an alternative way of viewing biframe morphisms. For biframes $\ca{L}$ and $\ca{M}$, let $C_L$ be the congruence on $\cpu{L}$ such that $(\cpu{L})$ is canonically isomorphic to $C_L\cong L$.
Let $C_M$ be the analogous congruence for $\ca{M}$. Using the defining property of biframe morphisms, we obtain that a suitable pair of frame maps $\pa{f}$ counts as a biframe morphism $\pa{f}:\ca{L}\ra \ca{M}$ if and only if there is a frame map $f:(\cpu{L})/C_L\ra (\cpu{M})/C_M$ making the following diagram commute.
\[
\begin{tikzcd}[row sep=large, column sep=large]
\cpu{L}
\arrow{r}{\cpu{f}}
\arrow{d}{q_{C_L}}
& \cpu{M}
\arrow{d}{q_{C_M}}
\\
(\cpu{L})/C_L
\arrow{r}{f}
& (\cpu{M})/C_M
\end{tikzcd}
\]
We proceed by giving the definition of the two adjoint functors we are interested in, as well as the definitions of the unit and the counit of the adjunction. 

\begin{itemize}
    \item For every bispace $X$, the biframe $\mf{b}\Om(X)$ is defined to be the triple 
    \[
    (\Om^+(X),\Om^-(X),\Om(X)),
    \]
    with set inclusions as the two frame injections. For a bicontinuous map $f:X\ra Y$, we define the map $\mf{b}\Om(f)$ as the pairing of preimage maps $(f^{-1},f^{-1}):\mf{b}\Om(Y)\ra \mf{b}\Om(X)$. This defines a functor $\mf{b}\Om:\bd{BiTop}\ra \bd{BiFrm}^{op}$.
    \item For every biframe $\ca{L}$, we define the bispace $\bpt(\ca{L})$ as follows. Its underlying set of points is $\bd{BiFrm}(\ca{L},\bd{2})$, where $\bd{2}$ denotes the biframe $(2,2,2)$. A positive open of $\bpt(\ca{L})$ is any subset of $\bd{BiFrm}(\ca{L},\bd{2})$ of the form $\{\pa{f}\in \bd{BiFrm}(\ca{L}, \bd{2}): f^+(a^+)=1\}$ for some $a^+\in L^+$; negative opens are defined analogously. For every biframe map $f:\ca{L}\ra \ca{M}$ we define the bicontinuous map $\bpt(f):\bpt(\ca{M})\ra \bpt(\ca{L})$ as the pairing of precompositions $(-\circ f^+,-\circ f^-)$. This defines a functor $\bpt:\bd{BiFrm}^{op}\ra \bd{BiTop}$, right adjoint to $\bOm$.
\end{itemize}

Now, let us look at the unit and the counit.
\begin{itemize}
    \item For every bispace $X$, we define the map $\mf{b}\psi_{X}:X\ra\bpt(\bOm(X))$ as $x\mapsto \pa{N_x}$ for each $x\in X$. Here, $N^+_x:\Om^+(X)\ra 2$ is defined to be $n^+_x(U^+)=1$ of and only if $x\in U^+$ for every positive open $U^+\se X$. We call this map the \tc{positive neighborhood map} of the bispace $X$. We call \tc{negative neighborhood map} the frame map $N^-_x:\Om^-(X)\ra 2$, defined analogously as $N^+_x$. The map $\mf{b}\psi_X:X\ra \bpt(\bOm(X))$ is always bicontinuous. We call this the \tc{sobrification map} of the bispace $X$. This determines a natural transformation $\mf{b}\psi:1_{\bd{BiTop}}\ra \bpt\circ \bOm$, and this is the unit of the adjunction.
    \item For every biframe $\ca{L}$ we define the map $\mf{b}\varphi_{\ca{L}}:\ca{L}\ra \bOm(\bpt(\ca{L}))$ as the pair of maps $\pa{\mf{b}\varphi_{\ca{L}}}$ where $\varphi_{\ca{L}}^+(a^+)=\{\pa{f}\in \bd{BiFrm}(\ca{L},\bd{2}):f^+(a^+)=1\}$ for every $a^+\in L^+$. The negative component of $\varphi_{\ca{L}}$ is defined similarly. We call $\mf{b}\varphi_{\ca{L}}$ the \tc{spatialization} of $\ca{L}$. This determines a natural transformation $\mf{b}\varphi:\bOm\circ \bpt\ra 1_{\bd{BiFrm}^{op}}$, and this is the counit of the adjunction.
\end{itemize}

\begin{theorem}
There is an adjunction $\bOm:\bd{BiTop}\lra  \bd{BiFrm}^{op}:\dpt$ with $\bOm\dashv\bpt$. The unit of adjunction is $\mf{b}\psi$ and the counit is $\mf{b}\varphi$.
\end{theorem}

\subsection{D-frames and their duality}
D-frames are quadruples $(L^+,L^-,c_L,t_L)$, where $L^+$ and $L^-$ are frames, and where $c_L,t_L\se \ltl$ satisfy certain axioms. We think of $L^+$ and $L^-$ as being two abstract topologies, $c_L$ as being an abstract disjointness relation, and $t_L$ as being an abstract covering relation. For two sets $L^+$ and $L^-$, we say that a family $\pa{x_i}\se \ltl$ is \tc{homogeneous} if it is either of the form $(x^+_i,x^-)$ or $(x^+,x^-_i)$. If $L^+$ and $L^-$ are posets, we say that a join (resp. a meet) in $\ltl$ is homogeneous when it is a join (resp. a meet) of a homogeneous family. We are now ready to give the formal definition of a d-frame.

A \tc{d-frame} is a quadruple $(L^+,L^-,c_L,t_L)$ such that $L^+$ and $L^-$ are frames, and such that $c_L$ and $t_L$ are subsets of $\ltl$ satisfying the following axioms.
\begin{enumerate}
    \item The subset $c_L$ is a downset of $\ltl$.
    \item The subset $c_L$ is closed under arbitrary homogeneous joins.
    \item The subset $t_L$ is an upset of $\ltl$.
    \item The subset $t_L$ is closed under finite homogeneous meets.
    \item Whenever we have $(a^+,a^-_1)\in c_L$ and $(a^+,a^-_2)\in t_L$, we must have $a^-_1\leq a^-_2$. Similar inequalities hold in the frame $L^+$.
\end{enumerate}

Axiom (5) is usually called \tc{balance}, and it is the pointfree version of the spatial fact that, whenever we have subsets $A,B_1,B_2\se X$ of some set $X$, if we have both $A\cap B_1=\emptyset$ and $A\cup B_2=X$, we must have that $B_1\se B_2$.

 In general, when we refer to a d-frame called $\ca{L}$, we implicitly say that its four components are called $(L^+,L^-,c_L,t_L)$. We refer to $c_L$ as the $\con$ component - or the consistency component - of the d-frame $\ca{L}$ and to $t_L$ as the $\tot$ component - or totality component - of $\ca{L}$. A morphism $f=(f^+,f^-):\ca{L}\ra \ca{M}$ in the category of d-frames consists of a pair of frame maps $f^+:L^+\ra M^+$ and $f^-:L^-\ra M^-$ such that they respect the $\con$ and the $\tot$ subsets, in the sense that $\pa{a}\in c_L$ implies that $\fpa{f}{a}\in c_M$, and similarly for the $\tot$ component. We call $\bd{dFrm}$ the category of d-frames with d-frame morphisms.

Let us now look at the dual adjunction between the category of bipaces and that of finitary biframes. We define the two functors first. 

\begin{itemize}
    \item For a bispace $X$, we 
define the d-frame $\dOm(X)$ as the quadruple
\[
(\Om^+(X),\Om^-(X),\{\pa{U}:U^+\cap U^-=\emptyset\},\{\pa{U}:U^+\cup U^-=X\}).
\]
 We extend this to a functor by defining the $\dOm$-image of a bicontinuous map $f:X\ra Y$ as the pair of preimage maps $(f^{-1},f^{-1}):\dOm(Y)\ra \dOm(X)$. 
    \item Let us call $\bd{2}$ the d-frame $(2,2,c_m,t_m)$. For clarity, we will denote the elements of the first component of this as $0^+<1^+$, and those of the negative frame component as $0^-<1^-$. For any d-frame $\ca{L}$, let us define as $\dpt(\ca{L})$ the collection of d-frame morphisms $f:\ca{L}\ra \bd{2}$. Now, let us equip this set with two topologies. Define a map $\phi_{\ca{L}}^+:L^+\ra \ca{P}(\dpt(\ca{L}))$ as 
\[
\phi_{\ca{L}}^+(a^+)=\{(f^+,f^-)\in \dpt(\ca{L}):f^+(a^+)=1^+\}.
\]
Define similarly the map $\phi^-_{\ca{L}}:L^-\ra \ca{P}(\dpt(\ca{L}))$. For a d-frame $\ca{L}$, define its \tc{d-spectrum}, or simply its \tc{spectrum}, as the bispace such that its underlying set of points is $\dpt(\ca{L})$, its positive opens are the sets of the form $\phi^+_{\ca{L}}(a^+)$ for some $a^+\in L^+$, and its negative opens are the sets of the form $\phi^-_{\ca{L}}(a^-)$ for some $a^-\in L^-$.
For a d-frame morphism $f:\ca{L}\ra \ca{M}$ we define $\dpt(f):\dpt(\ca{M})\ra \dpt(\ca{L})$ as the precomposition map $(-\circ f^+,-\circ f^-):\dpt(\ca{M})\ra  \dpt(\ca{L})$. 
\end{itemize}

Let us now define the unit and the counit of the adjunction.

\begin{itemize}
    \item For a bispace $X$, and a point $x$, define the map $N_x^+:\Om^+(X)\ra 2$ as $N^+_x(U^+)=1$ if and only if $x\in U^+$, and define similarly the negative neighborhood map $N^-_x:\Om^-(X)\ra 2$. For any bispace $X$ and any point $x\in X$, the unit of the adjunction assigns to every bispace $X$ the map 
\begin{align*}
 &   \mf{d}\psi_X:X\ra \dpt(\dOm(X))\\
 &      x\mapsto (N^+_x,N^-_x).
\end{align*}
    \item The counit of the adjunction assigns to every d-frame $\ca{L}$ the d-frame map
\begin{align*}
    & (\mf{d}\phi^+_{\ca{L}},\mf{d}\phi^-_{\ca{L}}):\ca{L}\ra \dOm (\dpt (\ca{L}))\\
    & a^+\mapsto \{\pa{f}\in \dpt(\ca{L}):f^+(a^+)=1^+\}.\\
    & a^-\mapsto \{\pa{f}\in \dpt(\ca{L}):f^-(a^-)=1^-\}.
\end{align*}
\end{itemize}

We have the following.
\begin{theorem}
There is an adjunction $\dOm:\bd{BiTop}\lra  \bd{dFrm}^{op}:\dpt$ with $\dOm\dashv\dpt$. The unit of adjunction is $\mf{d}\psi$ and the counit is $\mf{d}\varphi$.
\end{theorem}

\section{Finitary biframes and their duality}

We know that for every biframe $\ca{L}$ we have a frame isomorphism $i:L\cong(\cpu{L})/C_L$ for some congruence $C_L$ on $\cpu{L}$ such that the following commutes, as well as the analogous diagram for the negative component.
\[
\begin{tikzcd}[row sep=large, column sep=large]
L^+
\arrow[d,"q_{C_L}\circ \sy{-}",swap]
\arrow{dr}{e^+_L}
\\
(\cpu{L})/C_L
\arrow{r}{i}
& L
\end{tikzcd}
\]
Here, $q_{C_L}:\cpu{L}\ra (\cpu{L})/C_L$ is the canonical quotient map, and $\sy{-}:L^+\ra \cpu{L}$ the canonical injection of $L^+$ into the coproduct. For brevity, then, we denote the element $e^+_L(x^+)$ as $\syp{x}$ for every $x^+\in L^+$, motivated by the fact that $L$ is always isomorphic to some free construction $(\cpu{L})/C_L$, and that therefore we may always think of the generators as appearing in $L$ as syntactic expressions. This is however just a notational convention: for a biframe $\ca{L}$, even though we will always have an isomorphism $L\cong (\cpu{L})/C_L$ for some congruence $C_L$ on $\cpu{L}$, we will never assume without mention that this is an equality. Whenever we have a biframe $\ca{L}$ we call $\fin(L)$ the ordered collection of elements of $L$ such that they are finite joins of finite meets of generators, i.e. those of the form $(\syw{x_1}{x_1})\ve ...\ve (\syw{x_n}{x_n})$, for $x^+_m\in L^+$ and $x^-_m\in L^-$.

We also say that a congruence on $L$ is \tc{finitary} if and only if it is generated by a relation $R\se \fin(L)\times \fin(L)$. We say that a biframe $\ca{L}$ is \tc{finitary} if and only if its main component is such that it is canonically isomorphic to $(\cpu{L}/C_L)$ with $C_L$ a finitary congruence.   

\medskip

We start from describing the spectrum of a finitary biframe. The right adjoint giving raise to the desired duality is simply the bispectrum functor $\bpt:\bd{BiFrm_{fin}}^{op}\ra \bd{BiTop}$ restricted to the finitary biframes. We present an alternative way of describing the condition of a pair of maps $\pa{f}$ being a bipoint of a finitary biframe. For a finitary biframe $(L^+,L^-,L)$, we define a pair $\pa{f}\in \pt(L^+)\times \pt(L^-)$ to be a \tc{bipoint} if and only if whenever $a^+\we a^-\leq b^+\ve b^-$ in $L$ at least one of the following four conditions holds.
\[
\bcases{ 
f^+(a^+)=0\cr 
f^-(a^-)=0\cr
f^+(b^+)=1\cr
f^-(b^-)=1}
\]
Let us show that our definition of the bipoint of a finitary biframe is equivalent to the classical one of bipoint of a biframe.

\begin{lemma}
For a finitary biframe $(L^+,L^-,L)$, a pair $\pa{f}\in \pt(L^+)\times \pt(L^-)$ is a bipoint if and only if whenever $a^+\we a^-\leq b^+\ve b^-$ in $L$ at least one of the following four conditions holds.
\[
\bcases{ 
f^+(a^+)=0\cr 
f^-(a^-)=0\cr
f^+(b^+)=1\cr
f^-(b^-)=1}
\]
\end{lemma}
\begin{proof}
We identify points of $L$ with frame maps $\cpu{f}:\cpu{L}\ra 2$ that factor through $C_L$. We need to show that a pair the coproduct pairing $\cpu{f}:\cpu{L}\ra 2$ factors through $C_L$ if and only if the pair $\pa{f}$ satisfies the condition described in our claim. First, we will show that a pair $\pa{f}$ is a bipoint if and only if the pairing $\cpu{f}$ respects the inequalities of $(\cpu{L})/C_L$ of the form $\syw{a}{a}\leq \syv{b}{b}$. The pair $\pa{f}$ being a bipoint is equivalent to having that whenever $a^+\we a^-\leq b^+\ve b^-$ in the frame $L$ this implies $f^+(a^+)\we f^-(a^-)=0$ or $f^+(b^+)\ve f^-(b^-)=1$. This, in turn, is equivalent to having $f^+(a^+)\we f^-(a^-)\leq f^+(b^+)\ve f^-(b^-)$. By definition of the coproduct pairing of two frame maps, this is equivalent to $(\cpu{f})(\syw{a}{a})\leq (\cpu{f})(\syv{b}{b})$. The coproduct pairing $\cpu{f}$ respecting these inequalities is equivalent to it respecting all of the inequalities between finitary elements, since an inequality between finitary elements may always be rewritten as an inequality of the form $\bve_{x\leq m}a^+_x\we a^-_x\leq \bwe_{y\leq n}b^+_y\ve b^-_y$. 
\end{proof}

Let us describe our left adjoint. For a bispace $X$, we define its finitary frame of opens as $\Om_{\mf{fin}}(X)$, which we define to be quotient of the frame $\Om^+(X)\oplus \Om^-(X)$ by the relation 
\[
R_{\Om_{\fin}(X)}=\{(\syw{U}{U},\syv{V}{V}): U^+\cap U^-\se V^+\cup V^-\}.
\]
We then define its finitary biframe of opens pf $X$ as the triple 
\[
\bomf(X)=(\Om^+(X),\Om^-(X),\Om_{\fin}(X)).
\]
The triple is indeed a biframe. We have that the frame of the patch topology $\Om(X)$ is isomorphic to the quotient of $\cp{\Om^+(X)}{\Om^-(X)}$ by the relation encoding \tc{all} inclusions between patch-opens that hold in $X$, namely the relation
\[
R_{\Om(X)}=\{(\bve_i \syw{U_i}{U_i},\bve_j \syw{V_j}{V_j}:\bcu_i U^+_i\cap U^-_i\se \bcu_j V^+_j\se V^-_j\}.
\]
Because $(\Om^+(X),\Om^-(X),\Om(X))$ is a biframe, we know that the syntactic map $\sy{-}:\Om^+(X)\ra (\cp{\Om^+(X)}{\Om^-(X)})/R_{\Om(X)}$ is an injection, but since the relation $R_{\Om(X)}$ makes more identifications than $R_{\Om_{\fin}(X)}$ does, the syntactic map $\sy{-}:L^+\ra \Om_{\fin}(X)$ is an injection, too. Then, $\bomf(X)$ is a biframe for every bispace $X$.

\begin{remark}
Taking the finitary biframe of opens of a bispace keeps the information on all inclusions that hold in $X$ between finitary elements of the patch topology. As every inequality between finitary elements is of the form $\bcu_{x\leq m}U^+_x\cap U^-_x\se \bca_{y\leq n}V^+_y\cup V^-_y$, every such inequality reduced to a set of inequalities of the form $U^+\cap U^-\se V^+\cup V^-$.
\end{remark} On morphisms, for a bicontinuous map $f:X\ra Y$ we define $\bomf(f):\bomf(Y)\ra \bomf(X)$ to be pair of preimage maps $(f^{-1},f^{-1})$. Let us show that this is well-defined. To do this, we will use some more general results.

\begin{lemma}\label{relinclu9}
If $f:L\ra M$ is a frame map, and if $R$ is a relation on $L$ and $S$ a relation on $M$, and if $(x,y)\in R$ implies that $(f(x),f(y))\in S$, we have that there is a unique frame map $f_{RS}:L/R\ra M/S$ making the following diagram commmute.
\[
\begin{tikzcd}[row sep=large, column sep = large]
L
\arrow{r}{f} 
\arrow[swap]{d}{q_R} 
& M
\arrow{d}{q_S} \\
L/R  
\arrow{r}{f_{RS}} & 
M/S
\end{tikzcd}
\]
\end{lemma}
\begin{proof}
By the universal property of the quotient $L/R$, to prove our claim it suffices to show that the map $q_S\circ f$ respects the relation $R$. Indeed, if we have that $(x,y)\in R$, then by our assumption we have that $(f(x),f(y))\in S$, and so $q_S(f(x))\leq q_S(f(y))$. 
\end{proof}

For a finitary biframe $\ca{L}$, we call $R_L$ the relation $\{(a^+\we a^-,b^+\ve b^-): a^+\we a^-\leq b^+\ve b^-\mb{ in }L\}$ which generates the congruence $C_L$.

\begin{corollary}\label{pairismorbifrm91}
Suppose $\ca{L}$ and $\ca{M}$ are finitary biframes, and that we have frame maps $f^+:L^+\ra M^+$ and $f^-:L^-\ra M^-$. Suppose also that $(\syw{a}{a},\syv{b}{b})\in R_L$ implies that $((\cpu{f})(\syw{a}{a}),(\cpu{f})(\syv{b}{b}))\in R_M$. Then there is a unique frame map $f_{LM}$ making the following diagram commute
\[
\begin{tikzcd}[row sep=large, column sep = large]
\cpu{L}
\arrow{r}{\cpu{f}} 
\arrow[swap]{d}{q_L} 
& \cpu{M}
\arrow{d}{q_M} \\
L
\arrow{r}{f_{LM}} & 
M
\end{tikzcd}
\]
and this implies that $\pa{f}:\ca{L}\ra \ca{M}$ is a biframe map.
\end{corollary}
\begin{proof}
The first part of the statement follows directly from Lemma \ref{relinclu9}. For the second part of the statement, we recall that a pair of frame maps $\pa{f}$ is a biframe map $\pa{f}:\ca{L}\ra \ca{M}$ if and only if there is a frame map $f:L\ra M$ which restricts as $f^+$ on $L^+$ and as $f^-$ on $L^-$. If the assumptions in the statement hold, then indeed there is a frame map $f_{LM}:L\ra M$ which restricts on the frame components as required by commutativity of the diagram in the statement.
\end{proof}

\begin{corollary}\label{pairismorbifrm9}
For biframes $\ca{L}$ and $\ca{M}$, with $L$ canonically isomorphic to $(\cpu{L})/R_L$ and $M$ canonically isomorphic to $(\cpu{M})/R_M$, suppose that we have frame maps $f^+:L^+\ra M^+$ and $f^-:L^-\ra M^-$. If we have that 
\begin{align*}
    & (\syw{a}{a},\syv{b}{b})\in R_L\mb{ implies }\\
    & (\sy{f^+(a^+)}\we \sy{f^-(a^-)},\sy{f^+(b^+)}\ve \sy{f^-(b^-))}\in R_M
\end{align*}
then we have that $\pa{f}:\ca{L}\ra \ca{M}$ is a biframe map.
\end{corollary}
\begin{proof}
By definition of the coproduct pairing of two frame maps $f^+$ and $f^-$, the assumptions in Corollary \ref{pairismorbifrm91} amount to having that $(\syw{a}{a},\syv{b}{b})\in R_L$ implies $(\sy{f^+(a^+)}\we \sy{f^-(a^-)},\sy{f^+(b^+)}\ve \sy{f^-(b^-)})\in R_M$.
\end{proof}
We will use the corollary above several times, when we need to show that a certain pair of frame maps is a biframe map. Let us now use it to construct our functor $\bomf:\bd{BiTop}\ra \bd{BiFrm_{fin}}^{op}$. 

\begin{lemma}
For any bicontinuous map $f:X\ra Y$ between bispaces, the pair $(f^{-1},f^{-1}):\bomf(Y)\ra \bomf(X)$ is a biframe map.
\end{lemma}
\begin{proof}
Since $f$ is a bicontinuous map we have two frame maps $f^{-1}:\Om^+(Y)\ra \Om^+(X)$ and $f^{-1}:\Om^+(Y)\ra \Om^+(X)$. By Corollary \ref{relinclu9}, and by definition of the two relations $R_{\Om_{\fin}(X)}$ and $R_{\Om_{\fin}(Y)}$, to show that we have a biframe map $(f^{-1},f^{-1}):\bomf(Y)\ra \bomf(X)$ it suffices to show that we have 
\begin{align*}
    & C^+\cap C^-\se D^+\cup D^-\mb{ in }Y\mb{ implies }f^{-1}(C^+)\cap f^{-1}(C^-)\se f^{-1}(D^+)\cup f^{-1}(D^-)\mb{ in }X.
\end{align*}
 Since the preimage map is monotone, we have that $f^{-1}(C^+\cap C^-)\se f^{-1}(D^+\cup D^-)$. The desired result follows from the fact that the preimage map also preserves the lattice operations. We have then shown the desired inclusion of relations $R_{\Om_{\fin}(Y)}\se R_{\Om_{\fin}(X)}$.
\end{proof}

We then have the spectrum functor $\bpt:\bd{BiFrm_{fin}}^{op}\ra \bd{BiTop}$ and the finitary biframe of opens functor $\bomf:\bd{BiTop}\ra \bd{BiFrm_{fin}}^{op}$. Let us construct the unit and the counit of the desired adjunction. First, we look at the finitary biframe version of the sobrification of a space. For a bispace $X$, let us denote as $N^+_x$ and $N^-_x$ the neighborhood maps for the positive and the negative topology on $X$, respectively. Let us now define the frame map $N_x:\Om_{\fin}(X)\ra 2$ as the map defined on generators as $\syp{U}\mapsto N^+_x(U^+)$, and similarly on negative generators. To show that this is well-defined, it suffices to show that the map
\[
\cpu{N_x}:\cp{\Om^+(X)}{\Om^-(X)}\ra 2
\]
respects the relation $\{(\syw{U}{U},\syv{V}{V}):U^+\cap U^-\se V^+\cup V^-\}$. Let us show this as a lemma.
\begin{lemma}
For every bispace $X$ we have that $\pa{N_x}$ is a bipoint of $\bomf(X)$.
\end{lemma}
\begin{proof}
By definition of bipoint of a finitary biframe, it suffices to show that whenever $U^+\cap U^-\se V^+\cup V^-$ in $X$ at least one of the following conditions holds for every point $x\in X$. 
\[
\bcases{ 
N^+_x(U^+)=0\cr 
N^-_x(U^-)=0\cr
N^+_x(V^+)=1\cr
N^-_x(V^-)=1}
\]
We show that we reach a contradiction if we assume that none of these conditions holds. If we have a point $x\in X$ such that $N_x^+(U^+)=N^-_x(U^-)=1$ and that $N^+_x(V^+)=
N^-_x(V^-)=0$, by definition of the neighborhood maps this means that $x\in U^+\cap U^-$ but $x\notin V^+\cup V^-$, contradicting the set inclusion $U^+\cap U^-\se V^+\cup V^-$. We must then have that $\pa{N_x}$ is a bipoint of $\bomf(X)$.
\end{proof}
Then, for a bispace $X$ may define a map 
\begin{align*}
    & \psi_X:X\ra \bpt(\bomf(X))\\
    & x\mapsto \pa{N_x}.
\end{align*}
A typical positive open of $\bpt(\bomf(X))$ is a set of the form $\{\pa{f}\in \bpt(\bomf(X)):f^+(U^+)=1\}$ for some positive open $U^+\se X$. Its preimage under the map $\psi_X$, then, is a set of the form $\{x\in X:\pa{N_x}\mb{ is such that }N^+_x(U^+)=1\}$ which is just the set $U^+$ since $N^+_x(U^+)=1$ if and only if $x\in U^+$, by definition of the neighborhood map. Then, the map $\psi_X$ is bicontinuous for every bispace $X$. Let us show that the assignments $X\mapsto \psi_X$ can be pasted together to obtain a natural transformation.
\begin{proposition}
The assignment $\mf{Obj}(\bd{BiTop})\ra \mf{Mor}(\bd{BiTop})$ defined as $X\mapsto \psi_X$ is a natural transformation $1_{\bd{BiTop}}\ra \bpt\circ \bomf$.
\end{proposition}
\begin{proof}
Suppose that $f:X\ra Y$ is a bicontinuous map between bispaces. The naturality square we need to check is the following. We have used the fact that on morphisms $\bomf$ sends each bicontinuous map $f$ to the pair $(f^{-1},f^{-1})$, and $\bpt$ assigns to each pair of frame maps the respective precomposition map.
\[
\begin{tikzcd}[row sep=large, column sep = large]
X
\arrow{r}{\psi_X} 
\arrow[swap]{d}{f} 
&\bpt(\bomf(X))
\arrow{d}{(-\circ f^{-1},-\circ f^{-1})} \\
Y 
\arrow{r}{\psi_Y} & 
\bpt(\bomf(Y))
\end{tikzcd}
\]
Suppose that $x\in X$. We need to show that we have an equality of biframe maps $(N^+_x\circ f^{-1},N^-_x\circ f^{-1})=(N^+_{f(x)},N^-_{f(x)})$. Let us show that the two bipoints agree on the first component. Suppose that $V^+\in \Om^+(Y)$. We have the following chain of equivalences.
\begin{align*}
    & N^+_x(f^{-1}(U^+))=1\mb{ iff }\\
    & x\in f^{-1}(U^+)\mb{ iff }\\
    & f(x)\in U^+\mb{ iff }\\
    & N^+_{f(x)}(U^+)=1.\qedhere
\end{align*}
\end{proof}
Let us now look at the bispatialization of a finitary biframe. For a finitary biframe $\ca{L}$, let us define the map $\pa{\varphi_{\ca{L}}}:\ca{L}\ra \bomf(\bpt(\ca{L}))$ as the same pair of frame maps constituting the spatialization map for biframes. Let us check that the bispatialization of a finitary biframe is a biframe map.
\begin{lemma}
For any finitary biframe $\ca{L}$, the map $\pa{\varphi_{\ca{L}}}:\ca{L}\ra \bomf(\bpt(\ca{L}))$ is a biframe map.
\end{lemma}
\begin{proof}
We have seen already that both components of $\pa{\varphi_{\ca{L}}}$ are frame maps. Say that $L$ is canonically isomorphic to the coproduct $\cpu{L}$ quotiented by the relation $R_L$. We have that $\Om_{\fin}(\bpt(\ca{L}))$ is isomorphic to the coproduct $\Om^+(\bpt(\ca{L}))\times \Om^-(\bpt(\ca{L}))$ quotiented by the relation representing the inequalities of the form $U^+\cap U^-\se V^+\cup V^-$ in $\bpt(\ca{L})$. By Corollary \ref{relinclu9}, and by definition of the relation $R_{\Om_{\fin}(\bpt(\ca{L}))}$, we have to show that whenever we have $(\syw{a}{a},\syv{b}{b})\in R_L$ then we also have the subset inclusion $\va{L}{+}{a}\cap \va{L}{-}{a}\se \va{L}{+}{b}\cup \va{L}{-}{b}$ in $\bpt(\ca{L})$.
Suppose, then, that $(\syw{a}{a},\syv{b}{b})\in R_L$, that is, that $a^+\we a^-\leq b^+\ve b^-$ in $L$. 
Indeed, if $\pa{f}$ is a bipoint, and if both $f^+(a^+)=1$ and $f^-(a^-)=1$, then by definition of bipoint we must have that either $f^+(b^+)=1$ or $f^-(b^-)=1$. So, we have the desired subset inclusion in $\bpt(\ca{L})$.
\end{proof}
The lemma above shows that we have a well-defined assignment $\varphi:\mf{Obj}(\bd{BiFrm_{fin}})\ra \mf{Mor}(\bd{BiFrm_{fin}})$. Let us show that this is a natural transformation.
\begin{proposition}
The assignment $\mf{Obj}(\bd{BiFrm_{fin}})\ra \mf{Mor}(\bd{BiFrm_{fin}})$ defined as $\ca{L}\mapsto \varphi_{\ca{L}}$ is a natural transformation $1_{\bd{BiFrm}}\ra \bomf\circ \bpt$.
\end{proposition}
\begin{proof}
Suppose that $f:\ca{L}\ra \ca{M}$ is a biframe map and that $\ca{L}$ and $\ca{M}$ are finitary biframes. Naturality amounts to the commutativity of the following diagram in $\bd{Frm}$ and the analogous one for the negative component.
\[
\begin{tikzcd}[row sep=large, column sep = large]
\ca{L}
\arrow{r}{\varphi^+_{\ca{L}}} 
\arrow[swap]{d}{f^+} 
&\bomf(\bpt(\ca{L}))
\arrow{d}{\bomf(-\circ f^+)} \\
\ca{M} 
\arrow{r}{\varphi^+_{\ca{M}}} 
& \bomf(\bpt(\ca{M}))
\end{tikzcd}
\]
Suppose that $a^+\in L^+$. We have the following chain of equalities. We have used the fact that $\bomf$ acts on pairs of maps by assigning them to the corresponding pairs of preimage maps.
\begin{align*}
    & \bomf(-\circ f)(\varphi^+_{\ca{L}}(a^+))=\\
    & =\{g\in \bpt(\ca{M}): g\circ f\in \varphi^+_{\ca{L}}(a^+)\}=\\
    & =\{g\in \bpt(\ca{M}):g^+(f^+(a^+))=1\}=\\
    &=\varphi^+_{\ca{M}}(f^+(a^+)).\qedhere
\end{align*}
\end{proof}
Finally, let us check that the data that we have gathered up to now comes together to form an adjunction.

\begin{theorem}
There is an adjunction $\Om_{\mf{fin}}:\bd{BiFrm_{fin}}^{op}\lra \bd{BiTop}:\mf{bpt}$ with $\bomf\dashv \bpt$, with unit $\psi$ and counit $\varphi$.
\end{theorem}
\begin{proof}
Let us check that the two unit-counit triangles commute.
\begin{itemize}
    \item Let $\ca{L}$ be a biframe. First, we check that the following commutes in the category of bispaces.
\[
\begin{tikzcd}
{} 
& \bpt(\bomf(\bpt(\ca{L}))) 
\arrow{dr}{\bpt(\varphi_{\ca{L}})} \\
\bpt(\ca{L}) 
\arrow{ur}{\psi_{\bpt(\ca{L})}} 
\arrow{rr}{1_{\bpt(\ca{L})}} 
&& \bpt(\ca{L})
\end{tikzcd}
\]
Suppose that $\pa{f}\in \bpt(\ca{L})$. Recall that $\bpt$ maps a pair of frame maps to the respective precomposition pair. Then, commutativity of the triangle amounts to having $N^+_{\pa{f}}\circ \varphi^+_{\ca{L}}=f^+$, and similarly on the negative component. For $a^+\in L^+$, we have the following chain of equivalences.
\begin{align*}
    & N^+_{\pa{f}}(\va{L}{+}{a})=1\mb{ iff }\\
    & \pa{f}\in \va{L}{+}{a}\mb{ iff }\\
    & f^+(a^+)=1.
\end{align*}
We have obtained the desired equality of bipoints.
\item Let $X$ be a bispace. Commutativity of the second triangle amounts to commutativity of the following triangle in the category of frames, and the analogous one on the negative component.
\[
\begin{tikzcd}
{} 
& \bomf^+(\bpt(\bomf(X))) 
\arrow{dr}{\bomf(\psi_{X})^+} \\
\bomf^+(X) 
\arrow{ur}{\varphi^+_{\bomf(X)}} 
\arrow{rr}{1_{\bomf(X)}} 
&& \bomf^+(X)
\end{tikzcd}
\]
Recall that by definition of the functor $\bomf$ we have that $\bomf(\psi_X)^+=\psi_X^{-1}$. Suppose that $U^+$ is a positive open of the bispace $X$. We have the following chain of equalities. 
\begin{align*}
    & \psi_X^{-1}(\varphi^+_{\bomf(X)}(U^+))=\\
    & =\{x\in X:N^+_x\in \varphi^+_{\bomf(X)}(U^+))\}=\\
    & =\{x\in X:N^+_x(U^+)=1\}=\\
    & =U^+.\qedhere
\end{align*}
\end{itemize}
\end{proof}

We say that a bispace $X$ is \tc{bisober} if the unit $\psi_X$ is a bihomeomorphism. To avoid ambiguity with the biframe duality, we will refer simply as ``patch-sober" bispaces to those bispaces $X$ which are bihomeomorphic to the bispectrum of their biframe of opens $\bOm(X)$.
We say that a finitary biframe is \tc{bispatial} if the counit $\eta_{\ca{L}}$ is an isomorphism. We will explore these notions more in detail at the end of the next section.

\subsection{The finitary biframe duality compared with the d-frame and the biframe duality}

In this section, we will show that the category of finitary biframes is coreflective in $\bd{BiFrm}$. We will also see that the functor $\bomf$ taking a bispace to its finitary biframe of opens equals the composition $\mc{fin}\circ \bOm$ of the biframe of opens functor composed with the coreflector $\mc{fin}:\bd{BiFrm}\ra \bd{BiFrm_{fin}}$. We will relate the duality of finitary biframes in a similar way to the duality of d-frames, and show the precise sense in which taking the biframes of opens, finitary biframes of opens, and d-frames of opens are a series of functors in which each forgets more than the previous one.

We already know that we have a full inclusion $i:\bd{BiFrm_{fin}}\inclu \bd{BiFrm}$. Let us work towards defining the coreflector associated with this inclusion. For a biframe $\ca{L}$ and a congruence $C$ on $L$, we define its \tc{finitary interior} as the congruence on $L$ generated by 
\[
C\cap (\fin(L)\times \fin(L)),
\]
and we denote it as $\mc{fin}(C)$. We say that a relation on $L$ is \tc{finitary} if it is only constituted of pairs of finitary elements. Let us check the finitary interior indeed is an interior operator on the collection of all congruences on $L$.

\begin{lemma}
For any biframe $\ca{L}$, the map $\mc{fin}:\mf{A}(L)\ra \mf{A}(L)$ is an interior operator.
\end{lemma}
\begin{proof}
We need to check that the assignment is monotone, deflationary, and idempotent. Denote as $\overline{(-)}:\ca{P}(\ltl)\ra \ca{P}(\ltl)$ the closure operator mapping each relation on $L$ to the smallest congruence containing it. For monotonicity we notice that if we have two congruences $C\se D$ on $L$ then any pair of finitary elements in the first is in the second. Then $C\cap (\fin(L)\times \fin(L))\se D\cap (\fin(L)\times \fin(L))$. Since $\overline{(-)}$ is monotone, this implies $\mc{fin}(C)\se \mc{fin}(D)$. The assignment is also deflationary: for a congruence $C$, we have that $\mc{fin}(C)$ is generated by a relation $R\se C$, and by monotonicity and idempotency of $\overline{(-)}$ this implies $\mc{fin}(C)\se C$. Finally, let us check that $\mf{fin}$ is idempotent. For this, it suffices to show that for any congruence $C$ we have $\mc{fin}(C)\se \mc{fin}(\mc{fin}(C))$. We have that both $C\cap (\fin(L)\times \fin(L))\se \mc{fin}(C)$ and $C\cap (\fin(L)\times \fin(L))\se \fin(L)\times \fin(L)$, and so also 
\[
C\cap (\fin(L)\times \fin(L))\se \mc{fin}(C)\cap (\fin(L)\times \fin(L)).
\]
Taking the closure $\overline{(-)}$ on both sides yields $\mc{fin}(C)\se \mc{fin}(\mc{fin}(C))$.
\end{proof}
For our aims, given a biframe $\ca{L}$ we will be especially interested in the finitary interior of the congruence $C_L$ on $\cpu{L}$.
\begin{lemma}\label{alsorel9}
For a biframe $\ca{L}$, the congruence $\mc{fin}(C_L)$ on $\cpu{L}$ is that generated by the relation
\[
\mc{fin}(R_L)=\{(\syw{a}{a},\syv{b}{b}):a^+\we a^-\leq b^+\ve b^-\mb{ in }L\}
\]
\end{lemma}
\begin{proof}
Every finitary element of $L$ can be written either as a finite join of binary meets of generators or as a finite meet of binary joins of generators. The congruence $\mc{fin}(L)$, then, by definition is that generated by the relation
\[
\mc{fin}(R_L)'=\{(\bve_{x\leq m}\syw{a_x}{a_x},\bwe_{y\leq n}\syw{b_y}{b_y}):a^+_x\we a^-_x\leq b^+_y\ve b^-_y\}.
\]
By definition we then have that $\mc{fin}(R_L)\se \mc{fin}(R_L)'\se \mc{fin}(C_L)$. So, the congruence $\overline{\mc{fin}(R_L)}$ generated by $\mc{fin}(R_L)$ is included in $\mc{fin}(C_L)$. For the other set inclusion, we notice that by the frame laws any congruence which determines the inequalities represented by the relation $\mc{fin}(R_L)$ also determines those represented by the relation $\mc{fin}(R_L)'$, and so we also have $\mc{fin}(C_L)=\overline{\mc{fin}(R_L)'}\se \overline{\mc{fin}(R_L)}$.
\end{proof}
We keep the definition that we used in the lemma above, and for a biframe $\ca{L}$ we denote as $\mc{fin}(R_L)$ the relation $\{(\syw{a}{a},\syv{b}{b}):a^+\we a^-\leq b^+\ve b^-\mb{ in }L\}$ on the coproduct $\cpu{L}$.
For a biframe $\ca{L}$ we define as $\mc{fin}(L)$ the frame $(\cpu{L})/\mc{fin}(C_L)$. By Lemma \ref{alsorel9}, this is the same as the frame $(\cpu{L})/\mc{fin}(R_L)$. Finally, we define
\[
\mathit{fin}(\ca{L})=(L^+,L^-,\mathit{fin}(L)).
\]
We know that this is a biframe, as the map $\sy{-}:L^+\ra \mc{fin}(L)$ must be an injection: this holds because the congruence $\mc{fin}(C_L)$ in general identifies fewer elements than $C_L$ does, and we already know that the syntactic map $\sy{-}:L^+\ra (\cpu{L})/C_L\cong L$ is an injection, since by assumption $\ca{L}$ is a biframe.
Let us work towards defining $\mc{fin}$ on morphisms.

For a morphism $f:\ca{L}\ra \ca{M}$ of biframes, let us define the map $\mathit{fin}(f):\mathit{fin}(\ca{L})\ra \mathit{fin}(\ca{M})$ as being the same pair of frame maps. To see that this really is a biframe map, we need to once again appeal to Corollary \ref{pairismorbifrm9}. We need to show that there is a frame map $f:L\ra M$ making the following diagram commute.
\[
\begin{tikzcd}[row sep=large, column sep = large]
\cpu{L}
\arrow{r}{\cpu{f}} 
\arrow[swap]{d}{q_{\fin(C_L)}} 
& \cpu{M}
\arrow{d}{q_{\fin(C_M)}} \\
\mathit{fin}(L)
\arrow{r}{f} 
& \mathit{fin}(M)
\end{tikzcd}
\]
As usual, the vertical arrows are the two obvious quotients. By Corollary \ref{pairismorbifrm9}, to see that such a frame morphism exists, it suffices to know that if $(\syw{a}{a},\syv{b}{b})\in \mc{fin}(R_L)$ then we have $(\sy{f^+(a^+)},\sy{f^-(a^-)})\in \mc{fin}(R_M)$. Let us show this as a lemma.
\begin{lemma}
For a map $f:\ca{L}\ra \ca{M}$ of biframes, we have that $(\syw{a}{a},\syv{b}{b})\in \mc{fin}(R_L)$ implies $(\sy{f^+(a^+)},\sy{f^-(a^-)})\in \mc{fin}(R_M)$.
\end{lemma}
\begin{proof}
Suppose that $\pa{f}:\ca{L}\ra \ca{M}$ is a map of biframes, and that $(\syw{a}{a},\syv{b}{b})\in \mc{fin}(R_L)$, so that $\sy{a^+}\we\sy{ a^-}\leq \sy{b^+}\ve \sy{b^-}$ in $L$. Because $\pa{f}$ is assumed to be a biframe map, we know that $\sy{f^+(a^+)}\we \sy{f^-(a^-)}\leq \sy{f^+(b^+)}\ve \sy{f^-(b^-)}$ in $M$. By definition of $R_M$, then, we have the desired result. 
\end{proof}

Now, let us define the counit of our adjunction. For a biframe $\ca{L}$, we know that $\mc{fin}(C_L)$ makes fewer identifications than $C_L$ does, and this means that we a frame surjection $\mc{fin}(L)\ra L$, as $L$ is canonically isomorphic to $(\cpu{L})/C_L$. This surjection acts on positive generators as $\syp{a}\mapsto a^+$ and similarly on negative ones. This means that we have a map of biframes $i(\mc{fin}(\ca{L}))\ra \ca{L}$ which is the pair of the identity frame maps. We call this map $\varepsilon_{\ca{L}}:i(\mc{fin}(\ca{L}))\ra \ca{L}$ and take it to be the definition of the counit of our adjunction evaluated at $\ca{L}$.   

\begin{lemma}\label{squareistrivial}
The assignment $\varepsilon:\mf{Obj}(\bd{BiFrm})\ra \mf{Mor}(\bd{BiFrm})$ is a natural transformation.
\end{lemma}
\begin{proof}
Suppose that $f:\ca{L}\ra \ca{M}$ is a biframe morphism. Commutativity of the naturality square amounts to commutativity of the following one in $\bd{Frm}$, and the analogous one on the negative component. We have used the fact that the positive component of $i(\mc{fin}(\ca{L}))$ is $L^+$, and the fact that the functor $\mc{fin}$ takes a pair of frame morphisms to itself.
\[
\begin{tikzcd}[row sep=large, column sep = large]
L^+
\arrow{r}{\id^+_L} 
\arrow[swap]{d}{f^+} 
& L^+
\arrow{d}{f^+} \\
M^+
\arrow{r}{\id^+_M} 
& M^+
\end{tikzcd}
\]
The fact that this commutes is trivial.
\end{proof}

For every finitary biframe $\ca{L}$, the congruence $\mc{fin}(C_L)$ already is finitary, and so it is a fixpoint of the interior operator $\mc{fin}:\mf{A}(L)\ra \mf{A}(L)$. This means that it is its own finitary interior, and so we have an isomorphism of frames $L\cong \mc{fin}(L)$, acting on positive generators as $a^+\mapsto \syp{a}$, and similarly on the negative ones. This means that the pair of identity maps $(\id^+,\id^-):\ca{L}\ra \mc{fin}(i(\ca{L}))$ is always an isomorphism of biframes. We call this isomorphism $\eta_{\ca{L}}:\ca{L}\ra \mc{fin}(i(\ca{L}))$ and take this to be the definition of the unit of the adjunction evaluated at $\ca{L}$.

\begin{lemma}
The assignment $\eta:\mf{Obj}(\bd{BiFrm_{fin}})\ra \mf{Mor}(\bd{BiFrm_{fin}})$ is a natural transformation.
\end{lemma}
\begin{proof}
By an argument analogous to that of Lemma \ref{squareistrivial}, the naturality square amount to the commutativity of trivial squares in the category of frames.
\end{proof}
Finally, we prove the desired result that $\mi{fin}:\bd{BiFrm}\ra \bd{BiFrm_{fin}}$ is the desired coreflector. 
\begin{proposition}
We have an adjunction $\mathit{fin}:\bd{BiFrm}\lra \bd{BiFrm_{fin}}:i$ with $i\dashv \mathit{fin}$, with unit the $\eta$ and counit $\varepsilon$.
\end{proposition}
\begin{proof}
Let us show that the two unit-counit triangles commute.
\begin{itemize}
    \item Suppose that $\ca{L}$ is a finitary biframe. The first triangle is as follows.
  \[
 \begin{tikzcd}
 {} 
 & i(\mc{fin}(L))
\arrow{dr}{i(\varepsilon_{\ca{L}})} \\
i(\ca{L})
\arrow{ur}{\eta_{i(\ca{L})}} 
\arrow{rr}{1_{i(\ca{L})}} 
&& i(\ca{L})
\end{tikzcd}
\]
 Commutativity of the first triangle amounts to commutativity of the following triangle in $\bd{Frm}$ and the analogous one for the negative component.
   \[
 \begin{tikzcd}
 {} 
 & L^+
\arrow{dr}{i(\varepsilon_{\ca{L}})^+} \\
L^+
\arrow{ur}{\eta_{i(\ca{L})}^+} 
\arrow{rr}{\id^+_{i(\ca{L})}} 
&& L^+
\end{tikzcd}
\]
    By definition of the functors $\mc{fin}$ and $i$, and by definition of the unit and the counit of the adjunction, all three arrows in this diagram are the identity on $L^+$.
    \item Commutativity of the second triangle is trivial, too. Similarly as in the first item, its commutativity amounts to commutativity of two triangles in $\bd{Frm}$, and by definitions of our functors, and the unit and counit of the adjunctions, in both triangles all three arrows are identities.\qedhere
\end{itemize}
\end{proof}
\begin{corollary}
The category $\bd{BiFrm_{fin}}$ is coreflective in $\bd{BiFrm}$.
\end{corollary}

Having now proved that we have the desired adjunction, we may express the connection between the duality of finitary biframes and that of biframes and d-frames. The claim we wish to prove is that we have that $\bomf=\mc{fin}\circ \bOm$, and that $\dOm=\Delta\circ \mc{fin}\circ \bOm$, so that we may see the three bitopological dualities that we have seen in this thesis as being part of the same big picture, as follows.

\begin{tikzcd}[column sep=large, row sep=large,
]
{\bd{BiTop}} 
\arrow[r,"{\mf{b}\Om}"{name=omm},left, shift left=2.6]
& \bd{BiFrm}^{op}
\arrow[r,"\mathit{fin}"{name=fin},left, shift left=2.6]
\arrow[l,"\mf{bpt}"{name=ptt},left, shift left=2.6]
\arrow[phantom,from=ptt,to=omm,"\dashv" rotate=-90]
& \bd{BiFrm_{fin}}^{op}
\arrow[r,"\Delta"{name=k},left, shift left=2.6]
\arrow[l,""{name=inclu},left, shift left=2.6,hookrightarrow]
\arrow[phantom,from=ptt,to=omm,"\dashv" rotate=-90]
\arrow[phantom,from=fin,to=inclu,"\dashv" rotate=-90]
& {\bd{dFrm}^{op}} 
\arrow[l,"\Gamma"{name=idl},left, shift left=2.6]
\arrow[phantom,from=k,to=idl,"\dashv" rotate=-90]
\end{tikzcd}

We will see that in this diagram all arrows from $\bd{BiTop}$ to any of the other categories are the open sets functors relative to that duality, and - similarly - all arrows from any of the pointfree categories to $\bd{BiTop}$ are the spectrum functors.

\begin{proposition}\label{gammadelta}
We have that $\del\circ \mf{b}\Om=\dOm$ and $\mf{bpt}\circ \Gamma=\dpt$
\end{proposition}
\begin{proof}
The two frame components of the d-frames $\dOm(X)$ and $\del(\mf{b}\Om(X))$ are clearly the same. The pairs in $\con_{\mf{b}\Om(X)_{\we}}$ are the disjoint pairs, by definition of $\mf{b}\Om(X)$, and similarly one sees that the totality components of $\del(\mf{b}\Om(X))$ and $\dOm(X)$ coincide. For a morphism $f:X\ra Y$ we have that $\del(\mf{b}\Om(f))=\del((f^{-1},f^{-1},f^{-1}))=(f^{-1},f^{-1})=\dOm(f)$. For a biframe $\ca{L}$ and a pair $f^+:L^+\ra 2$ and $f^-:L^-\ra 2$, below, we denote as $f:L\ra 2$ the unique frame map that coincides with $f^+$ when restricted to $L^+$, and with $f^-$ when restricted to $L^-$. For any d-frame $\ca{L}$, we have the following.
\begin{align*}
    &\mf{bpt}(\Gamma(\ca{L}))=\\
    &=\{\pa{f}\in \pt(L^+)\times \pt(L^-): (f^+,f^-,f):\Gamma(\ca{L})\ra \bd{2}\mb{ is a biframe map}\}=\\
    &=\{\pa{f}\in\pt(L^+)\times \pt(L^-):f^+(a^+)\we f^-(a^-)=0\mb{ when $\pa{a}\in c_L$}\\
    & \mb{and }f^+(a^+)\ve f^-(a^-)=1\mb{ when $\pa{a}\in t_L$}\}=\\
    & \{\pa{f}\in\pt(L^+)\times \pt(L^-):f^+(a^+)=0\mb{ or } f^-(a^-)=0\mb{ when $\pa{a}\in c_L$,}\\
    & \mb{and }f^+(a^+)=1\mb{ or }f^-(a^-)=1\mb{ when $\pa{a}\in t_L$}\}=\\
    & =\dpt(\ca{L}).
\end{align*}
Showing that the positive and the negative topologies of the two bispaces $\mf{bpt}(\Gamma(\ca{L}))$ and $\dpt(\ca{L})$ are the same is easy; one does it by showing that this bijection is a bihomeomorphism. For a d-frame morphism $f:\ca{L}\ra \ca{M}$, we have that $\mf{bpt}(\Gamma(f))=\mf{bpt}((f^+,f^-,f))=(-\circ f,-\circ f)=\dpt(f)$.
\end{proof}

\begin{lemma}\label{finitaryconnect}
We have the following facts.
\begin{itemize}
    \item $\bomf=\mc{fin}\circ \bOm$.
    \item $\dOm=\Delta\circ \mc{fin}\circ \bOm$.
    \item $\dpt= \bpt \circ\Gamma$.
    \item For a d-frame $\ca{L}$, the biframe $\Gamma(\ca{L})$ is finitary.
\end{itemize}
\end{lemma}
\begin{proof}
Let us prove each item in turn.
\begin{itemize}
    \item To see that the two functors agree on objects, consider a bispace $X$. We have that $\mc{fin}(\bOm(X))$ is the biframe $(\Om^+(X),\Om^-(X),\mc{fin}(\Om(X)))$, but the frame $\mc{fin}(\Om(X))$ is by definition the same as the frame $\bomf(\Om(X))$. On morphisms, $\bOm$ and $\bomf$ both assign to each bicontinuous map $f$ the pairing $(f^{-1},f^{-1})$, and the functor $\mc{fin}$ leaves biframe maps unaltered.
    \item This follows from the first item and from Proposition \ref{gammadelta}.
    \item This is part of Proposition \ref{gammadelta}.
    \item For a d-frame $\ca{L}$, the biframe $\Gamma(\ca{L})$ has as its main component the relation
    \[
    \{(\syw{x}{x},0):\pa{x}\in c_L\}\cup \{(1,\syv{x}{x}):\pa{x}\in t_L\}
    \]
    on the coproduct $\cpu{L}$, and this is clearly a finitary relation.\qedhere
\end{itemize}
\end{proof}
Something stronger can be proven. We now would like to show that the adjunction $\dOm\dashv \dpt$ is the composition of the adjunction $\bOm_{\fin}\dashv \bpt$ and the adjunction $\Gamma\vdash \Delta$; and that similarly the adjunction $\bOm_{\fin}\dashv \bpt$ is the composition of the adjunction $\bOm\dashv \bpt$ and the adjunction $i\vdash \mi{fin}$. We denote an adjunction $F\vdash G$ with unit $\eta$ and counit $\varepsilon$ as the quadruple $(F,G,\eta,\varepsilon)$. It is a well-known fact that adjunctions can be composed, provided that their source and target categories agree. In particular, in \cite{maclane71} it is proven (Theorem 1.8 of Chapter IV) that whenever we have two adjunctions as follows

\begin{tikzcd}[column sep=large, row sep=large,
]
\ca{C}
\arrow[r,"F"{name=omm},left, shift left=2.6]
& \ca{D}
\arrow[r,"H"{name=fin},left, shift left=2.6]
\arrow[l,"G"{name=ptt},left, shift left=2.6]
\arrow[phantom,from=ptt,to=omm,"\dashv" rotate=-90]
& \ca{E}
\arrow[l,"I"{name=inclu},left, shift left=2.6]
\arrow[phantom,from=ptt,to=omm,"\dashv" rotate=-90]
\arrow[phantom,from=fin,to=inclu,"\dashv" rotate=-90]
\end{tikzcd}

such that the unit-counit pair of $F\dashv G$ is $(\eta^A,\varepsilon^A)$ and that of $H\dashv I$ is $(\eta^B,\varepsilon^B)$, respectively. We may compose them to obtain an adjunction
\[
(H,I,\eta^B,\varepsilon^B)\circ (F,G,\eta^A,\varepsilon^A)=(H\circ F,G\circ I,\eta^B\circ \eta^A,\varepsilon^B\circ \varepsilon^A).
\]
The unit and the counit of the composite adjunction are the horizontal compositions of the two original units and the two original counits, respectively. Explicitly, for an object $X\in \mf{Obj}(\ca{C})$ the unit is the morphism $G(\eta^B_{F(X)})\circ \eta^A_X$. For an object $Y\in \mf{Obj}(\ca{E})$, the counit of the composite adjunction is $\varepsilon^B_Y\circ H(\varepsilon^A_{I(Y)})$. Let us use the notion of composition of adjunctions to relate between each other the three bitopological spatial-sober adjunctions that we have seen. Below, the quadruples $(\bOm,\bpt,\mf{b}\psi,\mf{b}\varphi)$, $(\bomf,\bpt,\mf{bf}\psi,\mf{bf}\varphi)$, and $(\dOm,\dpt,\mf{d}\psi,\mf{d}\varphi)$ denote the obvious adjunctions.

\begin{theorem}
The adjunction $(\bomf,\bpt,\mf{bf}\psi,\mf{bf}\varphi)$ is the composition of adjunctions $(\bOm,\bpt,\mf{b}\psi,\mf{b}\varphi)\circ (\mc{fin},i,\eta,\varepsilon)$. The adjunction $(\dOm,\dpt,\mf{d}\psi,\mf{d}\varphi)$ is the composition of adjunctions $(\Delta,\Gamma,\delta,\gamma)\circ (\mc{fin},i,\eta,\varepsilon)\circ (\bOm,\bpt,\mf{b}\psi,\mf{b}\varphi)$.
\end{theorem}
\begin{proof}
By Lemma \ref{finitaryconnect}, it suffices to show that the compositions of the units and the counits give the desired result. First, consider a bispace $X$. We need to check that the following diagram commutes.
\[
\begin{tikzcd}[row sep=large, column sep = large]
X
\arrow{r}{\mf{bf}\psi_X}
\arrow{d}{\mf{b}\psi_X}
& \bpt(i(\mc{fin}(\bOm(X)))) \\
\bpt(\bOm(X)) 
\arrow[ur,swap,"\bpt(\varepsilon_{\bOm(X)})"] 
\end{tikzcd}
\]
We recall that for any biframe $\ca{L}$ the map $\varepsilon_{\ca{L}}:i(\mc{fin}(\ca{L}))\ra \ca{L}$ is the identity on both components, and that $\bpt$ acts on morphisms by sending each pair of frame maps to the corresponding pair of preimage maps, and so $\bpt(\varepsilon_{\ca{L}})=(\id^+,\id^-)$. Then, commutativity of the diagram amounts to having $\pa{N_x}=\pa{\id}\circ \pa{N_x}$ for every $x\in X$, and this is trivially true. Now for the counits of the first composition of adjunctions. Suppose that $\ca{L}$ is a finitary biframe. We need to check that the following diagram commutes.
\[
\begin{tikzcd}[row sep=large, column sep = large]
\ca{L}
\arrow{r}{\mf{bf}\varphi_{\ca{L}}}
\arrow{d}{\eta_{\ca{L}}}
& \mc{fin}(\bOm(\bpt(i(\ca{L})))) \\
\mc{fin}(i(\ca{L}))
\arrow[ur,swap,"\mc{fin}(\mf{b}\varphi_{i(\ca{L})})"] 
\end{tikzcd}
\]
We recall that the map $\eta_{\ca{L}}:\ca{L}\ra \mc{fin}(i(\ca{L}))$ is the identity map on both components, and that $\mc{fin}$ maps a pair of frame maps to itself. Then, commutativity of the triangle above amounts to having $\mf{b}\varphi_{i(\ca{L})}=\mf{bf}\varphi_{\ca{L}}$, and these holds by definition of these two maps.
Let us now check the second composition of adjunctions. By the results we have obtained so far, it suffices to show that $(\dOm,\dpt,\mf{d}\psi,\mf{d}\varphi)=(\Delta,\Gamma,\delta,\gamma)\circ (\bomf,\bpt,\mf{bf}\psi,\mf{bf}\varphi)$. Let us start with the units of the adjunctions. Suppose that $X$ is a bispace. We check that the following commutes.
\[
\begin{tikzcd}[row sep=large, column sep = large]
X
\arrow{r}{\mf{d}\psi_X}
\arrow{d}{\mf{bf}\psi_X}
& \bpt(\Gamma(\Delta(\bomf(X)))) \\
\bpt(\bomf(X)) 
\arrow[ur,swap,"\bpt(\gamma_{\bomf(X)})"] 
\end{tikzcd}
\]
We recall that for every biframe $\ca{L}$ the biframe map $\gamma:\Gamma(\Delta(\ca{L}))\ra \ca{L}$ is defined as the identity on both components. Commutativity of the diagram, then, amounts to the equality $\pa{\id}\circ \pa{N_x}=\pa{N_x}$ for every $x\in X$, which holds trivially. Finally, for a d-frame $\ca{L}$, we check that the following triangle commutes.
\[
\begin{tikzcd}[row sep=large, column sep = large]
\ca{L}
\arrow{r}{\mf{d}\varphi_{\ca{L}}}
\arrow{d}{\delta_{\ca{L}}}
& \Delta(\bomf(\bpt(\Gamma(\ca{L})))) \\
\Delta(\Gamma(\ca{L}))
\arrow[ur,swap,"\Delta(\mf{bf}\varphi_{\Gamma(\ca{L})})"] 
\end{tikzcd}
\]
We recall that for a d-frame $\ca{L}$ the d-frame map $\delta:\ca{L}\ra \Delta(\Gamma(\ca{L}))$ is defined as the identity on both frame components. The functor $\Delta$ maps each pair of frame maps to itself. Then, commutativity of the diagram amounts to having the equality $\mf{d}\varphi_{\ca{L}}=\mf{bf}(\varphi_{\ca{L}})$. This holds by definition of these two maps: we know that the target d-frame of opens is the same, and it is the d-frame of opens of $\dpt(\ca{L}))$ by Lemma \ref{finitaryconnect}, and each of these maps assigns to each $a^+\in L^+$ the set of pairs $\pa{f}$ with $f^+(a^+)=1$.
\end{proof}

From this we learn that the finitary biframe duality is a middle ground between the biframe duality and the d-frame duality. We have made precise the intuition that taking the finitary biframe of opens of a bispace preserves less information than taking its d-frame, but more information than taking its d-frame. To conclude this section, we look at some other more concrete facts relating bisobriety and bispatiality with the analogous notions for the biframe and the d-frame dualities.
\begin{lemma}\label{preinclu9}
We have the following facts.
\begin{itemize}
    \item For bispace $X$, we have a chain of bisubspace inclusions
    \[
  X\se \bpt(\bOm(X))\se \bpt(\bomf(X))\se \dpt(\dOm(X)),
    \]
    where the first only holds up to bihomeomorphism.
    \item For every biframe $\ca{L}$, we have a chain of biframe surjections
    \[
    \ca{L}\epi \bomf(\bpt(\ca{L}))\epi \bOm(\bpt(\ca{L})),
    \]
    and these are $\mf{bf}\varphi_{\ca{L}}$ and $\varepsilon_{\bOm(\bpt(\ca{L}))}$, respectively.
\end{itemize}
\end{lemma}
\begin{proof}
Let us prove the two items in turn.
\begin{itemize}
    \item By Lemma \ref{finitaryconnect}, it suffices to show that we have subspace inclusions
     \[
     X\se \bpt(\bOm(X))\se \bpt(\mc{fin}(\bOm(X))\se \bpt(\Gamma(\Delta(\mc{fin}(\bOm(X))).
     \]
     We know from the introduction that (up to isomorphism) we have a bisubspace inclusion $X\se \bpt(\bOm(X))$. For the other two inclusions, let us notice that each of the three spaces involved is defined to be the set of pairs $\pa{f}\in \pt(\Om^+(X))\times \pt(\Om^-(X))$ such that their coproduct pairing $\cpu{f}:\Om^+(X)\oplus \Om^-(X)\ra 2$ respects a certain relation. For the space $\bpt(\Om(X))$, the pair $\pa{f}$ needs to respect the relation $R_{\Om(X)}$ encoding all inequalities of the patch-opens of $X$. For the space $\bpt(\mc{fin}(\bOm(X)))$, it will have to only respect the finitary interior $\mc{fin}(R_{\Om(X)})$ of this relation. Finally, for the space $\bpt(\Gamma(\Delta(\mc{fin}(\bOm(X)))$ the pairing $\cpu{f}$ will only have to respect the relation $\Delta(R_{\Om(X)})$ encoding only the information on disjoint and covering pairs. As all pairs here are finitary, we have $\Delta(R_{\Om(X)})\se \mc{fin}(R_{\Om(X)})\se R_{\Om(X)}$, and so we have the chain of set inclusions
     \begin{align*}
         &   \{\pa{f}\in \pt(\Om^+(X))\times \pt(\Om^-(X)):\cpu{f}\mb{ respects }R_{\Om(X)}\}\se \\
         & \se \{\pa{f}\in \pt(\Om^+(X))\times \pt(\Om^-(X)):\cpu{f}\mb{ respects }\mc{fin}(R_{\Om(X)})\}\se\\
         & \se  \{\pa{f}\in \pt(\Om^+(X))\times \pt(\Om^-(X)):\cpu{f}\mb{ respects }\Delta(R_{\Om(X)})\}.
     \end{align*}

     The underlying sets of the three bispaces we are considering are then in the desired relation. The fact that these also are bisubspace inclusions follows from the fact that our three bitopologies are defined similarly.
    \item By definition of the topology on $\bpt(\ca{L})$, the bispatialization map $\varphi_{\ca{L}}$ is always a pair of frame surjections. We also have a frame surjection 
    \[
    (\Om^+(X)\oplus \Om^-(X))/\mc{fin}(R_{\Om(X)})\epi (\Om^+(X)\oplus \Om^-(X))/R_{\Om(X)})\cong \Om(X)
    \]
    as the relation $\mc{fin}(R_{\Om(X)})$ induces fewer identifications than $R_{\Om(X)}$. This means that the map $\varepsilon_{\bOm(\bpt(\ca{L}))}=\pa{\id}:\bomf(\bpt(\ca{L}))\ra \bOm(\bpt(\ca{L}))$ is a biframe surjection. Then, there is a biframe surjection $\bomf(\bpt(\ca{L}))\epi \bOm(\bpt(\ca{L}))$.\qedhere
\end{itemize}
\end{proof}

\begin{corollary}\label{composeadjconseq}
We have the following facts.
\begin{itemize}
    \item Every d-sober space is bisober, and every bisober space is patch-sober.
    \item If $\ca{L}$ is a spatial finitary biframe, then it is bispatial.
\end{itemize}
\end{corollary}
\begin{proof}
Let us prove the two items in turn.
\begin{itemize}
    \item By the first item of Lemma \ref{preinclu9}, if the bisubspace inclusion $X\se \dpt(\dOm(X))$ is an equality already, then so is $X\se \bpt(\bomf(X))$, and if $X\se \bpt(\bomf(X))$ is, then so is $X\se \bpt(\bOm(X))$.
    \item By the second item of Lemma \ref{preinclu9}, if the map $\mf{b}\varphi_{\ca{L}}:\ca{L}\epi \bOm(\bpt(\ca{L}))$, then so is the map $\mf{bf}\varphi_{\ca{L}}:\ca{L}\epi \bomf(\bpt(\ca{L}))$. \qedhere
\end{itemize}
\end{proof}
Let us now explore bispatiality and bisobriety taking into account these facts. We recall that given an adjunction $F:\ca{C}\lra \ca{D}:G$ with $F\dashv G$, unit $\eta$, and counit $\varepsilon$ we can describe the largest equivalence of categories which it restricts to. This equivalence relates the full subcategory $\mf{Fix}_{G\circ F}(\ca{C})$ of $\ca{C}$ determined by the objects $\{X\in \mf{Obj}(\ca{C}):\eta_X\mb{ is an isomorphism}\}$, and the analogous full subcategory of $\ca{D}$. We abuse notation by referring to this collection of objects, too, as $\mf{Fix}_{G\circ F}(\ca{C})$. We recall that the bispace-biframe adjunction is such that the set of objects of $\bd{BiFrm}$ which is a fixpoint of $\mf{b}\Om\circ \bpt$ coincides with $\{\ca{L}\in \mf{Obj}(\bd{BiFrm}):\ca{L}\cong \mf{b}\Om(X)\mb{ for some bispace }X\}$. Similarly, the bispaces which are fixpoints of $\bpt\circ \mf{b}\Om$ are those which are isomorphic to $\bpt(\ca{L})$ for some biframe $\ca{L}$.   
\begin{lemma}\label{functoriso}
Let $\ca{C}$ be a category. We have the following facts about isomorphisms.
\begin{itemize}
    \item If $F$ is a functor then $F(i)$ is an isomorphism for any isomorphism $i\in \mf{Mor}(\ca{C})$.
    \item If $i:X\cong Y$ and $j:Y\cong Z$ are isomorphisms in $\ca{C}$, then so is $j\circ i$, and its inverse is $i^{-1}\circ j^{-1}$.
\end{itemize}
\end{lemma}
\begin{proof}
For the first item, we observe that since functors respect identities and compositions we have that $F(i^{-1})\circ F(i)=F(i^{-1}\circ i)=F(1_{X})=1_{F(X)}$, and one shows similarly that $F(i)\circ F(i^{-1})=1_{F(Y)}$. For the second item, suppose that we have isomorphisms $i:X\cong Y$ and $j:Y\cong Z$ in $\ca{C}$. We have $i^{-1}\circ j^{-1}\circ j \circ i=i^{-1}\circ 1_{Y}\circ i=1_X$, and one shows similarly that $j\circ i\circ i^{-1}\circ j^{-1}=1_Z$.
\end{proof}
\begin{lemma}\label{fixiso9}
Suppose that we have an adjunction $F:\ca{C}\lra \ca{D}:G$ with $F\dashv G$, unit $\eta$, and counit $\varepsilon$. If $X\cong Y$ are isomorphic objects of $\ca{C}$ and if $X\in \mf{Fix}_{G\circ F}(\ca{C})$, then $Y\in \mf{Fix}_{G\circ F}(\ca{C})$, too. 
\end{lemma}
\begin{proof}
Suppose that we have $X,Y\in \mf{Obj}(\ca{C})$ and an isomorphism $i:X\cong Y$ in $\ca{C}$. Suppose also that $\eta_X:X\ra G(F(X))$ is an isomorphism. By \ref{functoriso}, we have an isomorphism $G(F(i)):G(F(X))\cong G(F(Y))$. Since $\eta$ is a natural transformation, furthermore, the following commutes.
\[
\begin{tikzcd}[row sep=large, column sep = large]
X
\arrow{r}{i} 
\arrow[swap]{d}{\eta_X} 
& Y
\arrow{d}{\eta_Y} \\
G(F(X)) 
\arrow{r}{G(F(i))} 
& G(F(Y))
\end{tikzcd}
\]
We claim that the inverse of $\eta_Y$ is $i\circ \eta_X^{-1}\circ G(F(i))^{-1}$. On the one hand, we can use commutativity of the diagram above and the second item of Lemma \ref{functoriso} to deduce that we have $\eta_Y\circ i\circ \eta_X^{-1}\circ G(F(i))^{-1}=G(F(i))\circ \eta_X\circ \eta_X^{-1}\circ G(F(i))^{-1}=G(F(i))\circ 1_{G(F(X))}\circ G(F(i))^{-1}=1_{G(F(Y))}$. On the other hand, we may use the same facts to also deduce $i\circ \eta_X^{-1}\circ G(F(i))^{-1}\circ \eta_Y=i\circ (G(F(i))\circ \eta_X)^{-1}\circ \eta_Y=i\circ (\eta_Y\circ i)^{-1}\circ \eta_Y=i\circ (\eta_Y\circ i)^{-1}\circ \eta_Y\circ i \circ i^{-1}=i\circ 1_{G(F(Y))}\circ i^{-1}=1_Y$. Then, indeed, $\eta_Y$ is an isomorphism.
\end{proof}

\begin{lemma}\label{composeandinclu9}
If two adjunctions $(F,G,\eta^A,\varepsilon^A)$ and $(H,I\eta^B,\varepsilon^B)$ -- composable as depicted below -- are such that we have the inclusion $\{F(X):X\in \mf{Obj}(\ca{C})\}\se \mf{Fix}_{F\circ G}(\ca{D})$ and the inclusion $\{H(X):X\in \mf{Obj}(\ca{D})\}\se \mf{Fix}_{H\circ I}(\ca{E})$, then we also have $\{H(F(X)):X\in \mf{Obj}(\ca{C})\}\se \mf{Fix}_{H\circ F\circ G\circ I}(\ca{E})$. An analogous result holds for objects of $\ca{C}$.

\begin{tikzcd}[column sep=large, row sep=large,
]
\ca{C}
\arrow[r,"F"{name=omm},left, shift left=2.6]
& \ca{D}
\arrow[r,"H"{name=fin},left, shift left=2.6]
\arrow[l,"G"{name=ptt},left, shift left=2.6]
\arrow[phantom,from=ptt,to=omm,"\dashv" rotate=-90]
& \ca{E}
\arrow[l,"I"{name=inclu},left, shift left=2.6]
\arrow[phantom,from=ptt,to=omm,"\dashv" rotate=-90]
\arrow[phantom,from=fin,to=inclu,"\dashv" rotate=-90]
\end{tikzcd}
\end{lemma}
\begin{proof}
Suppose that we have the two inclusions in the statement. Let $X\in \mf{Obj}(\ca{C})$. We need to show that the counit $\varepsilon^B\circ \varepsilon^A$ evaluated at $H(F(X))$ is an isomorphism. Explicitly, this is  $\varepsilon^B_{H(F(X))}\circ H(\varepsilon^A_{I(H(F(X)))})$. By assumption, the counit $\varepsilon^B_{H(F(X))}$ is an isomorphism, and so is the counit $\varepsilon^A_{I(H(F(X)))}$. By the first item of Lemma \ref{functoriso}, the morphism $H(\varepsilon^A_{I(H(F(X)))})$ is an isomorphism, too. Then, the counit $\varepsilon^B\circ \varepsilon^A$ evaluated at $H(F(X))$ is a composition of two isomorphisms, and this means that it is an isomorphism, by the second item of Lemma \ref{functoriso}.
\end{proof}

Consider an adjunction $F:\ca{C}\lra \ca{D}:G$ with $F\dashv G$, unit $\eta$, and counit $\varepsilon$. In general, we always have that $\mf{Fix}_{G\circ F}(\ca{C})\se \{X\in \mf{Obj}(\ca{C}):X\cong G(Y)\mb{ for some }Y\in \mf{Obj}(\ca{D})\}$. What we usually do not have is the reverse set inclusion. If for an adjunction this reverse set inclusion holds, too, for both categories, we say that the adjunction has the \tc{idempotency} property. Notice that, because of Lemma \ref{fixiso9}, in order for our adjunction $F\dashv G$ to have the idempotency property it suffices to have the inclusion $\{F(X):X\in \mf{Obj}(\ca{C})\}\se \mf{Fix}_{F\circ G}(\ca{D})$, and the analogous one for the other category. We call this property ``idempotency" because it implies that the composition $G\circ F$ is idempotent in the sense that we have an isomorphism $G(F(X))\cong G(F(G(F(X))))$ for every object $X\in \mf{Obj}(\ca{C})$, given by the unit $\eta_{G(F(X))}$. The idempotency property is equivalent to having that units of the form $\eta_{G(F(X))}$ and counits of the form $\varepsilon_{F(G(Y))}$ are all isomorphisms. 

\begin{lemma}\label{idempotencyadjsuff}
For an adjunction $(F,G,\eta,\varepsilon)$ with $F:\ca{C}\lra \ca{D}:G$, we have that if $\{G(Y):Y\in \mf{Obj}(\ca{D})\}\se \mf{Fix}_{G\circ F}(\ca{C})$, and if we have the similar subset inclusion for objects of $\ca{D}$, then the adjunction has the idempotency property.
\end{lemma}
\begin{proof}
If an adjunction satisfies the inclusion $\{G(Y):Y\in \mf{Obj}(\ca{D})\}\se \mf{Fix}_{G\circ F}(\ca{C})$ and the analogous one for the objects of the category $\ca{D}$, then whenever we have an object $X\in \mf{Obj}(\ca{C})$ which is isomorphic to $G(Y)$ for some $Y\in \mf{Obj}(\ca{D})$, by Lemma \ref{fixiso9} it is also in $\mf{Fix}_{G\circ F}(\ca{C})$. The argument for the inclusion for the category $\ca{D}$ is analogous.
\end{proof}

\begin{lemma}\label{composeidemadj}
Adjunctions with the idempotency property are closed under composition.
\end{lemma}
\begin{proof}
To see this, we combine Lemma \ref{idempotencyadjsuff} with Lemma \ref{composeandinclu9}.
\end{proof}

\begin{proposition}\label{finitaryadjidempotency}
The adjunction $\mf{b}\Om_{\fin}:\bd{BiTop}_\lra \bd{BiFrm_{fin}}^{op}:\bpt$ satisfies the idempotency property. In particular, the bispatial finitary biframes are exactly those of the form $\bomf(X)$ for a bispace $X$; and the bisober bispaces are exactly those of the form $\bpt(\ca{L})$ for some finitary biframe $\ca{L}$.
\end{proposition}
\begin{proof}
It suffices to show that the adjunction $\mc{fin}:\bd{BiFrm}\lra \bd{BiFrm_{fin}}:i$ has the idempotency property. We show this by appealing to Lemma \ref{composeidemadj}. In light of this lemma, to prove the desired claim it suffices to show that both the adjunctions $(\bOm,\bpt,\mf{b}\psi,\mf{b}\varphi)$ and $(\mc{fin},i,\eta,\varepsilon)$ have the idempotency property. We already know this to be true of the adjunction $(\bOm,\bpt,\mf{b}\psi,\mf{b}\varphi)$. For the other adjunction, we recall that the unit $\eta$ is defined as the pair of identity maps $\eta_{\ca{L}}:i(\mc{fin}(\ca{L}))\ra \ca{L}$ for every biframe $\ca{L}$. If $\ca{L}$ is of the form $i(\ca{L}')$, that it, if it is finitary, then this map is an isomorphism, as the finitary interior $\mc{fin}(C_L)$ is $C_L$ itself. We recall that the counit $\varepsilon$ is defined as the pair of identity maps $\varepsilon_{\ca{L}}: \ca{L}\ra \mc{fin}(i(\ca{L}))$ for any finitary biframe $\ca{L}$. We have observed already that this is always an isomorphism. In particular, all of the finitary biframes of the form $\mc{fin}(\ca{L})$ for some biframe $\ca{L}$ are such that $\eta_{\mc{fin}(\ca{L})}$ is an isomorphism. Then, the adjunction has the idempotency property.
\end{proof}

\subsection{Bisobriety and other separation axioms}\label{bisobriety}

We begin our more concrete exploration of bisobriety, and we show how it is related to other properties of bispaces. So far, we know by Corollary \ref{composeadjconseq} that it is stronger than patch-sobriety and weaker than bisobriety. Let us compare it with other separation axioms for bispaces.

\begin{lemma}\label{almostdsober}
If $X$ is a bispace and $\pa{f}$ is a bipoint of $\bomf(X)$, then 
\begin{itemize}
    \item if $U^+\cap U^-=\emptyset$, either $f^+(U^+)=0$ or $f^-(U^-)=0$,
    \item if $V^+\cup V^-=X$, either $f^+(U^+)=1$ or $f^-(V^-)=1$.
\end{itemize}

\end{lemma}
\begin{proof}
Suppose that $X$ is a bispace and that $\pa{f}\in \bpt(\bomf(X))$.
\begin{itemize}
    \item If $U^+\cap U^-\se \emptyset$, then by definition of bipoint at least one of the following four conditions must hold.
    \[
     \bcases{ 
     f^+(U^+)=0\cr 
     f^-(U^-)=0\cr
     f^+(\emptyset)=1\cr
     f^-(\emptyset)=1}
     \]
    The only ones which do not lead to contradiction are the first and the second, since for every frame map $f^+:\Om^+(X)\ra 2$ we must have that $f^+(\emptyset)=0$, and similarly on the negative component.
    \item If $X\se U^+\cup U^-$, then by definition of bipoint at least one of the following four conditions must hold.
     \[
     \bcases{ 
     f^+(X)=0\cr 
     f^-(X)=0\cr
     f^+(U^+)=1\cr
     f^-(U^-)=1}
     \]
    The only ones which do not lead to contradiction are the third and the fourth, since for every frame map $f^+:\Om^+(X)\ra 2$ we must have that $f^+(X)=1$, and similarly on the negative component.\qedhere
\end{itemize}
\end{proof}

We follow \cite{salbany74} in defining a bispace to be \tc{pairwise} $T_1$, or simply $\tc{bi}-T_1$, if whenever we have $x,y\in X$ with $x\neq y$ there is some $U\in \Om^+(X)\cup \Om^-(X)$ such that it contains $x$ and omits $y$. Still following \cite{salbany74}, we define bitopological space $X$ is said to be \tc{pairwise} $T_2$, or \tc{bi}-$T_2$, or \tc{pairwise Hausdorff}, if whenever we have $x,y\in X$ with $x\neq y$ there are disjoint opens $U^+\in \Om^+(X)$ and $U^-\in \Om^-(X)$ such that $x\in U^+$ and $y\in U^-$, or there are disjoint opens $U^+\in \Om^+(X)$ and $U^-\in \Om^-(X)$ such that $y\in U^+$ and $x\in U^-$.

The following proof is adapted from the proof of of Theorem 4.13 of \cite{jung06}, a result which states that a d-frame analogue of Hausdorffness implies d-sobriety. Notice how in that proof the classical notion of pairwise Hausdorffness needs to be strengthened, but replacing d-sobriety with bisobriety enables us to replicate the proof with pairwise Hausdorffness instead of its strong version.

\begin{lemma}\label{pairwisehausdorff}
Pairwise Hausdorff bispaces are bisober.
\end{lemma}
\begin{proof}
Suppose that $X$ is a pairwise Hausdorff space. Since no two points can have the same pair of neighborhood filters, injectivity of the bisobrification map is immediate. For surjectivity, suppose that $\pa{f}\in \bpt(\bomf(X))$ is a bipoint. 
We need to show that there is $x\in X$ with $\pa{f}$ as its neighborhood map. Let $\pa{P}\in \Om^+(X)\times \Om^-(X)$ be the pair of meet primes associated with $\pa{f}$, that is, let $P^+=\bve (f^+)^{-1}(0)$, and define similarly the negative open $P^-$. We cannot have that $P^+\cup P^-=X$, or by Lemma \ref{almostdsober} we could not have that $f^+(P^+)=f^-(P^-)=0$. Then, let $x\in (P^+\cup P^-)^c$. We will first show that $(P^+\cup P^-)^c=\{x\}$, and then that $\pa{f}$ is the neighborhood map of $x$. Suppose towards contradiction that there is some $y\in (P^+\cup P^-)^c$ with $y\neq x$. By pairwise Hausdorffness, without loss of generality we may assume that there are disjoint $U^+$ and $U^-$ with $x\in U^+$ and $y\in U^-$. By Lemma \ref{almostdsober}, we must have that either $f^+(U^+)=0$ or $f^-(U^-)=0$. These facts imply, respectively, that $U^+\se P^+$ and that $U^-\se P^-$, but both these are contradictions with witnesses $x$ and $y$, respectively. Then, we must indeed have that $(P^+\cup P^-)^{c}=\{x\}$. Let $\pa{N_x}$ be the pair of neighborhood maps of $x$, and let $\pa{P_x}$ be the pair of primes associated with it. We have $P^+\se P^+_x$ and $P^-\se P^-_x$. We want to show the remaining inequalities $P^+_x\se P^+$ and $P^-_x\se P^-$. Since $P^+\cup P^-=X{\sm}\{x\}$ we must have that $(X\cap P^+_x)\se P^+\cup P^-$. Since $\pa{f}$ is a bipoint, we have that one of the following four conditions hold.
\[
\bcases{ 
f^+(P_x^+)=0\cr 
f^-(X)=0\cr
f^+(P^+)=1\cr
f^-(P^-)=1}
\]
The only one which does not lead to contradiction is $f^+(P^+_x)=0$, which by definition of $P^+$ implies $P^+_x\se P^+$. Similarly one shows $P^-_x\se P^-$.
\end{proof}

We now show that bisobriety is incomparable with pairwise $T_1$. The following example is adapted from the counterexample given in \cite{johnstone82} that not all $T_1$ spaces are sober. There, it is noticed that the space with underlying set $\mathbb{N}$ equipped with the cofinite topology is $T_1$ but it is not sober.

\begin{example}
Consider the bispace $(\mathbb{N},\mf{CoFin}(\mathbb{N}),\mathcal{I})$,
where $\mf{CoFin}(\mathbb{N})$ is the cofinite topology on $\mathbb{N}$ and $\mathcal{I}$ is the indiscrete topology. Since the positive topology is $T_1$ already, the space is bi-$T_1$. On the other hand, the patch topology is $(\mathbb{N},\mf{CoFin}(\mathbb{N}))$, and this is not sober. Since a bisober bispace is always patch-sober by Corollary \ref{composeadjconseq}, the bispace $(\mathbb{N},\mf{CoFin}(\mathbb{N}),\mathcal{I})$ cannot be bisober.
\end{example}

We will use the following lemma to find a counterexample for the converse implication.

\begin{lemma}\label{veryeasypatch}
If $X=(X,\Om^+(X),\ca{I})$ is a bispace where $\mathcal{I}$ is the indiscrete topology, its finitary biframe of opens is $(\Om^+(X),2,\Om^+(X))$. Its bisobrification is bihomeomorphic the bispace $(\pt(\Om^+(X)),\Om(\pt(\Om^+(X))),\mathcal{I})$.
\end{lemma}
\begin{proof}
Since $2$ is the initial object in the category of frames, it is the neutral element of the frame coproduct. Then, for a bispace $(X,\Om^+(X),\ca{I}$ we have that its finitary frame of opens is $\Om^+(X)\oplus 2\cong \Om^+(X)$ quotiented by the inclusions of finitary subsets in $X$. As $\ca{I}$ is the indiscrete topology, its opens are also opens in $X$, and so there are no subset inclusions of finitary opens in $X$ which are not inclusions of positive opens. Then, the finitary biframe of our bispace is indeed isomorphic to $(\Om^+(X),2,\Om^+(X))$. To calculate the bispectrum of this, we notice that its bipoints are in bijection with the points of the main component, i.e. the underlying set of points is $|\pt(\Om^+(X))|$. A typical positive open of this bispace is $\{f\in \pt(\Om^+(X)):f^+(U^+)=1\}$ for some $U^+\in \Om^+(X)$, and so the positive topology is homeomorphic to the topology of $\pt(\Om^+(X))$. As for the negative topology, we notice that the only opens are $\{f\in \pt(\Om^+(X)):f(1)=1\}=\pt(\Om^+(X))$ and $\{f\in \Om^+(X):f(1)=0\}=\emptyset$. 
\end{proof}
\begin{example}
For a natural number $n\in \mathbb{N}$ with $1<n$, consider the bispace \[
(\{1,...,n\},\{\emptyset\}\cup \{\da m\leq n\},\mathcal{I}).
\]
This is not a bi-$T_1$ space: for any $m$ such that $1<m\leq n$ we have that there is no open in any of the two topologies which contains $m$ and omits $1$. However, it is a bisober space: by Lemma \ref{veryeasypatch} its biframe of opens is isomorphic to $(C_{n+1},2,C_{n+1})$, where $C_{n+1}$ is the $n+1$-elements chain. The bispectrum of this, by Lemma \ref{veryeasypatch}, is bihomeomorphic to the starting bispace.
\end{example}

Finally, it is clear that if a bispace $X$ is bisober then it is also pairwise $T_0$. If a bispace $X$ is bisober, in particular the bisobrification map $x\mapsto \pa{N_x}$ must be injective, and so distinct points must differ by their positive or their negative open nieghborhood filter.

\section{The assembly of a finitary biframe}

For any frame $L$ and for any subset $S\se L$ following \cite{Wilson1994TheAT} we define the frame $\mf{A}_S(L)$ is the subframe of $\mf{A}(L)$ generated by the collection
\[
\{\na(x):x\in L\}\cup \{\del(s):s\in S\}
\]
For a biframe $\ca{L}$, we abbreviate as $\mf{A}_{\fin}(L)$ the frame $\mf{A}_{\fin(L)}(L)$, and we call it as the \tc{finitary assembly} of $L$. Notice that the definition of the finitary assembly depends on the bitopological structure of $\ca{L}$. In \cite{schauerte92}, it is proven that

\begin{observation}\label{frmclo6}
Suppose that $L$ is a frame. The map $\mathit{frm}:\ca{P}(L)\ra \ca{P}(L)$, assigning to each subset of $L$ the subframe generated by it, is a closure operator.
\end{observation}

\begin{lemma}\label{manysubframes6} For any biframe $\ca{L}$, the following are all descriptions of the finitary assembly of a frame $L$.
\begin{enumerate}
    \item The subframe of $\mf{A}(L)$ generated by all closed congruences of $L$ and the open ones of the form $\del (\sy{x^+}\ve \sy{x^-})$.
    
     \item The subframe of $\mf{A}(L)$ generated by the collection
    \[
    \{\na (\sy{x^+}):x^+\in L^+\}\cup \{\del (\sy{x^-}):x^-\in L^-\}\cup \{\na (\sy{x^-}):x^-\in L^-\}\cup \{\del (\sy{x^+}):x^+\in L^+\}.
    \]
    
    \item The subframe of $\mf{A}(L)$ generated by $\sy{\mf{A}(L^+)}\cup \sy{\mf{A}(L^-)}$.

    \item The ordered collection of congruences of $L$ of the form $ \bve_i \na (\sy{x^+_1}\we \sy{x^-_i})\cap\del (\sy{y^+_1}\ve \sy{y^-_i})$.
   
\end{enumerate}
\end{lemma}
\begin{proof}
Let us prove each item in turn. For a collection $C\se L$, let us denote as $\mathit{frm}(C)$ the subframe of $\SLop$ generated by $C$.
\begin{enumerate}

\item The finitary assembly of $L$ is the subframe of $\SLop$ generated by the closed congruences of $L$ together with the finitary open ones. Since $C_1=\{\na (x):x\in L\}\cup \{\del (x^+\ve x^-):x^+\in L^+,x^-\in L^-\}$ is a subset of this collection, we clearly have $\mathit{frm}(C_1)\se \mf{A}_{\fin}(L)$. For the other subframe inclusion, we first claim that whenever $f\in L$ is a finitary element, we have that $\del (f)\in \mathit{frm}(C_1)$. This holds because opens $\del (f)$ from finitary elements $f\in L$ are finite joins in $\mf{A}(L)$ of open congruences of the form $\del (\sy{x^+}\ve \sy{x^-})$. Then, we may deduce that we have the inclusion $\{\na(x):x\in L\}\cup \{\del (f):f\in \fin(L)\}\se \mathit{frm}(C_1)$. Clearly this also implies the subframe inclusion $\mathit{frm}(\{\na(x):x\in L\}\cup \{\del (f):f\in \fin(L)\})\se \mathit{frm}(C_1)$, that is, $\mf{A}_{\fin}(L)\se \mathit{frm}(C_1)$.

\item Let us denote as $S^+$ the collection $\{\del (\sy{x^+}):x^+\in L^+\}\cup \{\na(\sy{x^+}):x^+\in L^+\}$ and let us denote as $S^-$ the similar collection of elements if $L^-$. It is clear that we have $S^+\cup S^-\se \{\na(x):x\in L\}\cup \{\del (f):f\in \fin(L)\}$, and so we immediately get the subframe inclusion $\mathit{frm}(S^+\cup S^-)\se \mf{A}_{\fin}(L)$. For the other set inclusion, we show that for every $x\in L$ we have that $\na(x)\in \mathit{frm}(S^+\cup S^-)$. For this, we recall that the map $\na:L\ra \mf{A}(L)$ is a frame map, and so for every arbitrary element $\bve_i (\sy{x^+_i}\we \sy{x^-_i})\in L$ its closed congruence is indeed a join of finite meets of closed congruences of the form $\na(x^+)$ and $\na(x^-)$. Then, $\{\na(x):x\in L\}\se \mathit{frm}(S^+\cup S^-)$. We now show that every $\del (f)$, for $f\in L$ finitary, is in $\mathit{frm}(S^+\cup S^-)$. For this, it suffices to notice that every such open congruence is the finite meet of finite joins of open congruences in $S^+\cup S^-$. Then, we have $\{\na(x):x\in L\}\cup \{\del (f):f\in \fin(L)\}\se \mathit{frm}(S^+\cup S^-)$, and so we are done.

\item Since the map $\mf{A}(i^+):\mf{A}(L^+)\ra \mf{A}(L)$ is a frame map, and since we know that open and closed sublocales congruences any assembly as a frame, we have that
\[
\sy{\mf{A}(L^+)}=\mathit{frm}(\{\del (\sy{x^+}):x^+\in L^+\}\cup \{\na (\sy{x^+}):x^+\in L^+\})=\mathit{frm}(S^+).
\] 
We have a similar result for the assembly of the negative frame $L^-$.  Substituting this into the equality we wish to claim, we obtain that this is equivalent to 
\[
\mf{A}_{\fin}(L)=\mathit{frm}(\mathit{frm}(S^+)\cup\mathit{frm}(S^-)).
\]
We use the fact that $\mathit{frm}$ is a closure operator, which is Observation \ref{frmclo6}, to deduce that this equality  amounts to $\mf{A}_{\fin}(L)=\mathit{frm}(S^+\cup S^-)$. This is item (2).

\item Let us denote as $F$ the collection of congruences of $L$ of the form $ \bve_i \na (\sy{x^+_1}\we \sy{x^-_i})\cap \del (\sy{y^+_1}\ve \sy{y^-_i})$. To prove the desired claim, we will show that $F=\mathit{frm}(\{\na(x):x\in L\}\cup \{\del(f):f\in \fin(L)\})$. Let us now define
\[
F'=\{\na (\sy{x^+}\we \sy{x^-}):x^+\in L^+,x^-\in L^-\}\cup \{\del (\sy{x^+}\ve \sy{x^-}):x^+\in L^+,x^-\in L^-\}
\]
Since every element in $F$ is the join of elements in $F'$, and since clearly $F'\se \{\na(x):x\in L\}\cup \{\del(f):f\in \fin(L)\}$, we obtain the chain of inclusions $F\se \mathit{frm}(F')\se \Sfop$. To prove the reverse inclusion, we first notice that the collection $\{\na(x):x\in L\}$ is a subframe of $\mf{A}(L)$ and that it is generated by the collection $\{\na(\sy{x^+}\we \sy{x^-}):x^+\in L^+,x^-\in L^-\}$. Additionally, every open of the form $\del (f)$ for $f\in L$ a finitary element is the finite meets of elements of the form $\del (\sy{x^+}\ve \sy{x^-})$. We then have the chain of inclusions $\mf{A}_{\fin}(L)=\mathit{frm}(\{\na(x):x\in L\}\cup \{\del (f):f\in \fin(L)\})\se \mathit{frm}(F')$. Closing the family $F$ under finite meets yields the family $\{\na(\sy{x^+}\we \sy{x^-})\ve \del(\sy{y^+}\ve \sy{y^-}):x^+,y^+\in L^+,x^-,y^-\in L^-\}$. Closing this under arbitrary joins clearly yields $F$, by definition of $F$, and so $\mathit{frm}(F')=F$. \qedhere
\end{enumerate}
\end{proof}

With this section, we move on to looking at the lattice of quotients of a finitary biframe. The first step is to look at the lattice of the finitary congruences on $L$ for some biframe $\ca{L}$.  

\begin{lemma}\label{AfinLisSL}
For a biframe $\ca{L}$, the ordered collection of finitary congruences of $L$ a subframe of $\mf{A}(L)$. In particular, it is $\mf{A}_{\fin}(L)$.
\end{lemma}
\begin{proof}
A congruence of $L$ is finitary exactly when it is generated by a relation of the form $\{(\syw{a_i}{a_i},\syv{b_i}{b_i})\}$, that is, if it is of the form 
\[
\bve_i \na(\syw{a_i}{a_i})\cap \del(\syv{b_i}{b_i})=\bve_i \na(\sy{a^+_i})\cap \na(\sy{a^-_i})\cap \del(\sy{b^+_i})\cap \del(\sy{b^-_i}),
\]
in which we have used the fact that $\na:L\ra \mf{A}(L)$ preserves the lattice operations and $\del:L\ra \mf{A}(L)$ reverses them. By the last item of Lemma \ref{manysubframes6}, the one we obtained is exactly the canonical form of a finitary congruence of $L$.
\end{proof}

Let us now take the second step of our analysis: using the notion of finitary congruence to define the notion of ``biquotient" of a finitary biframe. A quotient of a biframe $\ca{L}$ is a biframe of the form $(L^+/C^+,L^-/C^-,L/C)$, where $C$ is a congruence on $L$, and $C^+$ and $C^-$ are the congruences on $L^+$ and $L^-$ corresponding to the restriction of $C$ to the subframes $\syp{L}$ and $\sym{L}$, respectively. We denote this quotient as $\ca{L}/C$. For a relation $R$ on $L$, we denote as $\ca{L}/R$ the biframe quotient $(L^+/(\overline{R})^+,L^-/(\overline{R})^-,L/R)$, where $\overline{R}$ is the congruence on $L$ generated by $R$. For a finitary biframe $\ca{L}$ and a relation $R\se L\times L$, we call $\ca{L}\ra \ca{L}/R$ a \tc{biquotient} if the relation $R$ induces a finitary congruence on $L$. For a finitary biframe $\ca{L}$, we denote as $\mf{S}(\ca{L})$ the ordered collection of biquotients of $\ca{L}$, under the order defined as $\ca{L}/C\leq \ca{L}/D$ if and only if $D\se C$ for any two congruences $C,D\se L\times L$. We now see that biquotients are just biframe quotients such that the target is a finitary biframe.

\begin{lemma}\label{witness}
Suppose that we have a frame $L$, a relation $R$ on $L$, and a relation $S$ on $L/R$. Suppose that for every $R$-equivalence class we have a selected witness $w([x]_R)\in [x]_R$ of the equivalence class. Let $w(S)=\{(w([a]_R),w([b]_R)):([a]_R,[b]_R)\in S\}$. We have that
\[
(L/R)/S=L/(R\cup w(S)).
\]
\end{lemma}
\begin{proof}
To prove our statement, we show that the frame maps $f:L\ra M$ which factors through one quotients coincide with those that factor through the other. The set of morphisms $f:L\ra M$ which factor through the quotient $(L/R)/S$ is the set
\[
\{f\in \bd{Frm}(L,M): f\mb{ respects }R,f(a)\leq f(b)\mb{ whenever }([a]_R,[b]_R)\in S\}.
\]
Now, suppose that for every $R$-equivalence class $[x]_R$ we select a witness $w([x]_R)\in [x]_R$. We have that the set of morphisms which factors through the quotient $L/(R\cup w(S))$ is
\begin{align*}
    & \{f\in \bd{Frm}(L,M): f\mb{ respects }R,f(w([a]_R))\leq f(w[b]_R)\mb{ whenever }([a]_R,[b]_R)\in S\}.
\end{align*}

The membership condition of the first set is formally stronger than the second, as we require $f(a')\leq f(b')$ for all elements $a'\in [a]_R$ and all elements $b'\in [b]_R$ for every pair of equivalence classes $([a]_R,[b]_R)\in S$, and not just for selected witnesses. So, the first set is included in the second. For the reverse inclusion, suppose that $f$ is in the second set. We then have that, whenever $a,b\in L$ are such  that $([a]_R,[b]_R)\in S$, we have 
\[
f(a)=f(w([a]_R))\leq f(w([b]_R))=f(b),
\]
where for the two equalities we have used the fact that $f$ respects $R$. So, the map $f:L\ra M$ also satisfies the membership condition for the first set.
\end{proof}

\begin{proposition}\label{whenbiquotient}
For a finitary biframe $\ca{L}$ we have that a biframe quotient $f:\ca{L}\epi \ca{L}/R$ is a biquotient if and only if the biframe $\ca{L}/R$ is finitary.
\end{proposition}
\begin{proof}
For simplicity, we will identify $L$ with its isomorphic copy $(\cpu{L})/C_L$. Suppose that $\ca{L}$ is a finitary biframe, and that we have a biframe quotient $\ca{L}/R$. First, suppose that $R$ induces a finitary congruence on $(\cpu{L})/C_L$. Let this finitary congruence be that induced by 
\[
\{(\sy{\syp{a_i}}\we \sy{\sym{a_i}},\sy{\syp{b_i}}\ve \sy{\sym{b_i}}):i\in I\}.
\]
Now, for every equivalence class $[\syp{a}]_{C_L}=\sy{\syp{a}}$ we select the witness $\syp{a}\in [\syp{a}]_{C_L}$, and similarly for the negative component. By Lemma \ref{witness}, we have that $\ca{L}/R$ is the same as the quotienting of $(L^+,L^-,\cpu{L})$ by the relation
\[
C_L\cup \{(\syp{a_i}\we \sym{a_i},\syp{b_i}\ve \sym{b_i}):i\in I\}
\]
on $\cpu{L}$. Since $C_L$ is a finitary congruence, then, this relation induces a finitary congruence on $\cpu{L}$. So, the biframe $\ca{L}/R$ is finitary. For the converse, suppose that the biframe $\ca{L}/R$ is finitary. Then, we have that $\ca{L}/R$ is obtained by quotienting $(L^+,L^-,\cpu{L})$ by some finitary congruence $C$ larger than $C_L$. We write this finitary congruence in canonical form, as the join
\[
C=\bve_i \na(\sy{a^+_i})\cap \na(\sy{a^-_i})\cap \del(\sy{b^+_i})\cap \del(\sy{b^-_i})
\]
of all congruences of the form $\na(a^+)\cap \na(a^-)\cap \del(b^+)\cap \del(b^-)$ below it. Since the congruence $C_L$ is finitary (by our initial assumption) and is contained in $C$, we must have that $C_L=\bve _{j\in J}\na(\sy{a^+_j})\cap \na(\sy{a^-_j})\cap \del(\sy{b^+_j})\cap \del(\sy{b^-_j})$ for some $J\se I$. Now, we may write
\[
C=C_L\ve \bve _{i\in I{\sm}J}\na(\sy{a^+_i})\cap \na(\sy{a^-_i})\cap \del(\sy{b^+_i})\cap \del(\sy{b^-_i}).
\]
If we set that every element of the form $\syp{a}$ as a witness for an equivalence class $[\syp{a}]_{C_L}=\sy{\syp{a}}$ once again, by Lemma \ref{witness} we obtain that the biframe $\ca{L}/R$ is the biframe $(L^+,L^-,L)$ quotiented by the relation
\[
\bve _{i\in I{\sm}J}\na(\sy{\sy{a^+_i}})\cap \na(\sy{\sy{a^-_i}})\cap \del(\sy{\sy{b^+_i}})\cap \del(\sy{\sy{b^-_i}}).
\]
This congruence must be finitary, and by definition of the biframe $\ca{L}/R$, this congruence is also that generated by $R$.
\end{proof}

With the next proposition, we examine the bispectra of biquotients of a finitary biframe $\ca{L}$. For a relation $R$ on $L$, we regard the bispectrum $\bpt(\Lq{R})$ as the set of points of $\ca{L}$ which factors through this quotients, equipped with the bitopology inherited from $\bpt(\ca{L})$.

\begin{proposition}\label{manyfacts9}
For a finitary biframe $\ca{L}$, we have the following facts.
\begin{enumerate}
    \item $|\bpt(\ca{L}/\del(\syp{a}))|=\va{L}{+}{a}$.
    \item $|\bpt(\Lq{\na(\syp{a})})|=\va{L}{+}{a}^c$.
    \item $|\bpt(\Lq{(\na(\syp{a})\cap \na(\sym{a})\cap \del(\syp{b})\cap \del(\sym{b}))})|=\va{L}{+}{a}^c\cup \va{L}{-}{a}^c\cup \va{L}{+}{b}\cup \va{L}{-}{b}$.
    \item $|\bpt(\Lq{\bve_i C_i})|=\bca_i |\bpt(\Lq{C_i})|$ for any collection $C_i$ of finitary congruences on $L$.
    \item $|\bpt(\Lq{C})|=\bca\{\va{L}{+}{a}^c\cup \va{L}{+}{b}\cup \va{L}{-}{a}^c\cup \va{L}{-}{b}:\less{a^+}{b^+}\cap \less{a^-}{b^-}\se C\}$ for any finitary congruence $C$ on $L$.
\end{enumerate}
\end{proposition}
\begin{proof}
Let us prove each item in turn. 
\begin{enumerate}
\item The underlying set of $\pt(\Lq{\del(\syp{a})})$ consists of all those bipoints $\pa{f}:\ca{L}\ra 2$ such that the map $f:L\ra 2$ generated by $f^+$ and $f^-$ respects the congruence $\del(\syp{a})$. This is the congruence generated by the relation $\{(1,\syp{a})\}$, and so for $\pa{f}$ to be in $\bpt(\Lq{\del(\syp{a})})$ it suffices for it to respect this relation, that is, it suffices to have $f^+(a^+)=1$. Then, $|\bpt(\Lq{\del(\syp{a})})|=\{\pa{f}\in \bpt(\ca{L}):f^+(a^+)=1\}=\va{L}{+}{a}$.
\item The argument is similar to the one for the first item. For $\syp{a}\in L$, we notice that the points of $\bpt(\Lq{\na(\syp{a})})$ is the set of bipoints $\pa{f}$ with $f^+(a^+)=0$, and that this is $\va{L}{+}{a}^c$.
\item Since the map $\na:L\ra \mf{A}(L)$ preserves the lattice operations, and the map $\del:L\ra \mf{A}(L)$ reverses them, we have $C:=\na(\syp{a})\cap \na(\sym{a})\cap \del(\syp{b})\cap \del(\sym{b})=\na(\syw{a}{a})\cap \del(\syv{b}{b})$. Furthermore, we know that in general a congruence $\na(a)\cap \del(b)$ is the congruence generated by the relation $\{(a,b)\}$. So the congruence we are considering on $L$ is that generated by the relation $\{(\syw{a}{a},\syv{b}{b})\}$. The points of $\bpt(\Lq{C})$ are those bipoints $\pa{f}$ such that the map $f:L\ra 2$ satisfies $f(\syw{a}{a})\leq f(\syv{b}{b})$, that is, such that either $f(\syw{a}{a})=0$ or $f(\syv{b}{b})=1$. As in $2$ the top element is join prime and the bottom one is meet prime, this condition amounts to having either $f^+(a^+)=0$, or $f^-(a^-)=0$, or $f^+(b^+)=1$, or $f^-(b^-)=1$. By definition of the bitopology on $\bpt(\ca{L})$, the set of points satisfying this is $\va{L}{+}{a}^c\cup \va{L}{-}{a}^c\cup \va{L}{+}{b}\cup \va{L}{-}{b}$.
\item If $C_i$ is a collection of finitary congruences on $L$, then the bipoints of $\bpt(\Lq{\bve_i C_i})$ are the collection of bipoints $\pa{f}\in \bpt(\ca{L})$ such that the map $f:L\ra 2$ respects all of the congruences in $\{C_i:i\in I\}$.
\item This follows from items (3) and (4). \qedhere
\end{enumerate}
\end{proof}

We would now like to show that for any finitary biframe $\ca{L}$ every bisober bisubspace of $\bpt(\ca{L})$ is of the form $\bpt(\ca{L}/C)$ for some finitary congruence $C$ on $L$, that is, that any bisober bisubspace of $\bpt(\ca{L})$ is the bispectrum of some biquotient of $\ca{L}$. To do this, we appeal to very general categorical facts.

\begin{lemma}\label{simplelemmaiso3}
Suppose that $\ca{C}$ is a category, and $i:X\ra Y$ is an isomorphism in $\ca{C}$, and $j:Y\ra X$ is such that $j\circ i=1_{X}$. Then $i\circ j=1_Y$.
\end{lemma}
\begin{proof}
Since we know that $i:X\ra Y$ is an isomorphism, we know that an inverse $j':Y\ra X$ exists. Then, we must have that $j'\circ i=1_X=j\circ i$. Since all isomorphisms are epimorphisms, this implies that $j=j'$ and, in particular, that $i\circ j=1_Y$.
\end{proof}

 \begin{lemma}\label{unitbecomesiso3}
 Suppose that we have an adjunction $F:\ca{C}\lra \ca{D}:G$ with $F\dashv G$, with $\eta:1\ra G\circ F$ the unit and $\varepsilon:F\circ G\ra 1$ the counit. Suppose also that this restricts to an equivalence of the full subcategories determined by $\{F(X):X\in \mf{Obj}(\ca{C})\}$ and $\{G(Y):Y\in \mf{Obj}(\ca{D})\}$. Then, for every $X\in \mf{Obj}(\ca{C})$ and every object $Y\in \mf{Obj}(\ca{D})$ the maps $F(\eta_X)$ and $G(\varepsilon_Y)$ are isomorphisms.
 \end{lemma}
 \begin{proof}
 Suppose that $X\in \mf{Obj}(\ca{C})$. We know by assumption that the counit $\varepsilon_{F(X)}$ is an isomorphism. Additionally, the following unit-counit triangle commutes.
\[
\begin{tikzcd}
{} & F(G(F(X))) \arrow{dr}{\varepsilon_{F(X)}} \\
F(X) \arrow{ur}{F(\eta_{X})} \arrow{rr}{1_{F(X)}} && F(X)
\end{tikzcd}
\]
Then, by Lemma \ref{simplelemmaiso3}, we must have that $F(\eta_X)\circ \varepsilon_{F(X)}=1_{F(G(F(X)))}$. Since it is the inverse of an isomorphism, the map $F(\eta_X)$ must be an isomorphism too. The argument for maps of the form $G(\varepsilon_Y)$ is analogous.
 \end{proof}

\begin{lemma}\label{bisobersarequotients}
For a finitary biframe $\ca{L}$, any bisober bisubspace inclusion $\bpt(\ca{M})\inclu \bpt(\ca{L})$ is, up to bihomeomorphism, and inclusion of the form $\bpt(\Lq{C})\se \bpt(\ca{L})$ for some finitary congruence $C$ on $L$. 
\end{lemma}
\begin{proof}
Consider an inclusion of the form $i:\dpt(\ca{M})\inclu \dpt(\ca{L})$ for some finitary biframe $\ca{M}$. Dualizing this we obtain a surjection of finitary biframes $\Om_{\fin}(i):\Om (\bpt(\ca{L}))\epi \Om_{\fin} (\bpt(\ca{M}))$, and so by Proposition \ref{whenbiquotient} there is also a biquotient $q_{R'}:\Om_{\fin} (\bpt(\ca{L}))\epi \Om_{\fin} (\dpt(\ca{L}))/R'$ relative to this surjection. Let us denote as $q_{\phi}:\ca{L}\epi \ca{L}/R_{\phi}$ the bispatialization quotient of $\ca{L}$, and note that $\ca{L}/R_{\phi}$ is isomorphic to $\Om_{\fin}(\bpt(\ca{L}))$. We then may compose these two quotients to obtain the quotient $q_{R'}\circ q_{\phi}: \ca{L}\epi \ca{L}/R$, with $R=R'\cup R_{\phi}$. Dualizing this, we obtain the bisubspace inclusion $\bpt(q_{R'}\circ q_{\phi}): \bpt(\Lq{R})\inclu \bpt(\ca{L})$ and, by functoriality of $\bpt$, this is $\bpt(q_{R'})\circ \bpt(q_{\phi})$. Since $\bpt(q_{\phi})$ is a bihomeomorphism by Lemma \ref{unitbecomesiso3}, this bisubspace inclusion is bihomeomorphic to $\dpt(q_{R'}):\bpt(\Lq{R'})\inclu \bpt(\ca{L})$ and, by definition of the quotient $q_{R'}$, this is also bihomeomorphic to the inclusion $i:\bpt(\ca{M})\inclu \bpt(\ca{L})$.
\end{proof}

This gives us a way of describing an arbitrary bisober bisubspace of a spectrum $\bpt(\ca{L})$ for some finitary biframe $\ca{L}$.
\begin{corollary}\label{generalbisobersub}
If $\ca{L}$ is a finitary biframe, the bisober bisubspaces of $\bpt(\ca{L})$ are exactly those whose underlying sets are of the form 
\[
\bca_i \va{L}{+}{a_i}^c\cup \va{L}{+}{b_i}\cup \va{L}{-}{a_i}^c\cup \va{L}{-}{b_i}
\]
for some $a^+,b^+\in L^+$ and some $a^-,b^-\in L^-$.
\end{corollary}
\begin{proof}
The bisober bisubspaces of $\bpt(\ca{L})$, by Lemma \ref{bisobersarequotients}, may be identified with the subspaces of $\bpt(\ca{L})$ of the form $\bpt(\Lq{C})$ for some finitary congruence $C$ on $L$. The desired result follows from item (5) of Proposition \ref{manyfacts9}.  
\end{proof}

For a bispace $X$, we define its \tc{Skula bispace} as the bispace with the same underlying set of $X$, whose positive opens are the topology generated by the positive opens and the negative closed sets of $X$, and whose negative opens are the topology generated by the negative opens and the positive closed sets of $X$. We denote the Skula bispace of $X$ as $Sk(X)$. Our new notion of bitopological sobriety enables us to prove a very close analogue of the classical result that the sober subspaces of a sober space coincide with those whose underlying set is Skula-closed. We stress that the analogue of this result is false for both the classical theory of biframes and the theory of d-frames.

\begin{theorem}\label{skulabisobers}
For any finitary biframe $\ca{L}$, the bisober bisubspaces of $\bpt(\ca{L})$ are exactly those such that their underlying sets are patch-closed sets of $Sk(\bpt(\ca{L}))$.
\end{theorem}
\begin{proof}
This follows from Corollary \ref{generalbisobersub} and from the definition of Skula bispace.
\end{proof}

\subsection{The assembly functor for finitary biframes and its spatial counterpart}

Recall that for a biframe $\ca{L}$ we define its assembly to be the structure
\[
\mf{A}(\ca{L})=(\mf{A}_{\na^+\del^-}(L),\mf{A}_{\na^-\del^+}(L),\mf{A}_{\fin}(L))
\]
where $\mf{A}_{\na^+\del^-}(L)$ is the subframe of $\mf{A}_{\fin}(L)$ generated by the collection $\{\na(a^+):a^+\in L^+\}\cup \{\del(a^-):a^-\in L^-\}$ and $\mf{A}_{\na^-\del^+}(L)$ is defined similarly. This structure is introduced in \cite{schauerte92}, where it is also shown that there is a canonical biframe embedding $\na=\pa{\na}:\ca{L}\ra \mf{A}(\ca{L})$ assigning to each $a^+$ the corresponding closed congruence $\na(a^+)$, and defined similarly on the negative component. We begin this section by showing that the assembly of a biframe is isomorphic to a certain free construction. From now on, we will abbreviate the filter completion functor $\mf{Filt}:\bd{Distr}\ra \bd{Frm}$ as simply $\mf{F}$. We will show that the assembly of a biframe $\ca{L}$ is isomorphic to the biframe quotient
\renewcommand{\Fi}{\mf{F}}
\renewcommand{\Fip}{\mf{F}(L^+)}
\renewcommand{\Fim}{\mf{F}(L^-)}
\[
(L^+\oplus \Fim,L^-\oplus \Fip,L^+\oplus \Fim\oplus L^-\oplus \Fip)/(C_L\cup \mc{Com}^+_L\cup \mc{Com}^-_L),
\]
where name ``$\mc{Com}^+_L$" stands for ``complementation", and denotes the relation 
\begin{align*}
    & \{(\syp{a}\we \syp{\up a},0),(1,\syp{a}\ve\syp{\up a}):a^+\in L^+\}
\end{align*}
and the relation $\mc{Com}^-_L$ is defined similarly for negative generators $a^-\in L^-$. We will denote as $\mf{A}^m(\ca{L})$ the biframe $(L^+\oplus \Fim,L^-\oplus \Fip,L^+\oplus \Fim\oplus L^-\oplus \Fip)$. Above, we have slightly abused notation by regarding $C_L$ as a congruence on $\mf{A}^m(\ca{L})$ rather than on $\cpu{L}$. Let us check that this structure is indeed isomorphic to the biframe assembly, by showing that these structures share the same universal property.  Let us recall that for a biframe $\ca{L}$ we have the notion of \tc{bipseudocomplement}. The bicomplement of a generator $x^+\in L^+$ is denoted as $\sim x^+$, and it is $\bve \{x^-\in L^-:\syw{x}{x}=0\mb{ in }L\}$. Bipseudocomplements of negative generators are defined similarly. When a bipseudocomplement $\sim x^+$ is such that $\syv{x}{x}=1$ in $L$, the element $\sim x^+$ is called a \tc{bicomplement} of $x^+$. Recall that in \cite{schauerte92} it is shown the assembly of a biframe $\ca{L}$ has the universal property that whenever we have a biframe map $f:\ca{L}\ra \ca{M}$ such that it provides complements for all elements of $L^+\cup L^-$ there is a unique biframe map $\tilde{f}:\Ad\ra \ca{M}$ making the following diagram commute.
\[
\begin{tikzcd}
{} 
& \Ad
\arrow{dr}{\tilde{f}} \\
\ca{L}
\arrow{ur}{\na} 
\arrow{rr}{f} 
&& \ca{M}
\end{tikzcd}
\]
In order to show the claim that the biframe assembly is isomorphic to the free construction introduced above, it will suffice to show that this free constructions has this universal property. 

\begin{lemma}\label{isalatticemap9}
Suppose that $f:\ca{L}\ra \ca{M}$ is a biframe map such that it provides complements to all elements of $L^+\cup L^-$. Then, the maps $\sim f^+:L^+\ra M^-$ and $\sim f^-:L^-\ra M^+$ are lattice maps.
\end{lemma}
\begin{proof}
First, we notice that the map $\sim f^+:L^+\ra M^-$ is well-defined by definition of bicomplementation. We show that this map turns finite meets of $L^+$ into finite joins of $L^-$, as the argument for the reversing of finite joins is analogous. For every $x^+\in L^+$, we have that $\sy{f^+(x^+)}$ and $\sy{\sim f^+(x^+)}$ are complements of each other in $M$. Recall that complemented elements of a frame form a sublattice, and that - when restricted to this sublattice - complementation reverses the lattice operations. This means that whenever $x^+_n\in L^+$ we have that the elements $\sy{f^+(x^+_1)}\we ...\we \sy{f^+(x_n^+)}$ and $\sy{\sim f^+(x^+_1)}\ve...\ve \sy{\sim f^+(x_n^+)}$ are complements of each other in $M$. As $\sim x_1^+\ve...\ve \sim x_n^+\in L^-$, we have that the elements $f^+(x_1^+)\we...\we f^+(x_n^+)$ and $\sim f^+(x_1^+)\ve... \sim f^+(x^+_n)$ are bicomplements of each other. Since bicomplements are bipseudocomplements, this implies the desired result, as we have the chain of equalities $\sim f^+(x_1^+\we...\we x_n^+)=\sim (f^+(x_1^+)\we...\we f^+(x_n^+))=\sim f^+(x_1^+)\ve...\ve \sim f^+(x_n^+)$.
\end{proof}

\begin{proposition}\label{assemblyiso}
For a biframe $\ca{L}$ we have an isomorphism of biframes
\[
\Ad\cong (L^+\oplus \Fim,L^-\oplus \Fip,L^+\oplus \Fim\oplus L^-\oplus \Fip)/(C_L\cup \mc{Com}^+_L\cup \mc{Com}^-_L),
\]
defined on positive generators as 
\begin{align*}
    & \na(x^+)\mapsto \sy{\syp{x}},\\
    & \del(x^-)\mapsto \sy{\sym{\up a}},
\end{align*}
and defined similarly on negative generators.
\end{proposition}
\begin{proof}
To show that the two biframes are isomorphic, we show that they share the same universal property. Suppose, then, that there is a biframe map $f:\ca{L}\ra \ca{M}$. For a lattice map $h:L_1\ra L_2$ from a distributive lattice to a frame, we denote as $\overline{h}:\Id(L_1)\ra L_2$ the frame map given by the universal property of the ideal completion. To show that there is a map as desired, we will show that the four frame maps
\begin{align*}
    & f^+:L^+\ra M^+\\
    & \overline{\sim f^-}:\Fi(L^-)\ra M^+\\
    & f^-:L^-\ra M^-\\
    & \overline{\sim f^+)}:\Fi(L^-)\ra M^+
\end{align*}
are such that the pairing $(f^+\oplus\overline{\sim f^-},f^-\oplus \overline{\sim f^+)}$ respects the relation $\mc{Com}^+_L\cup \mc{Com}^-_L\cup \sy{C_L}$. We know that the second and fourth map indeed are frame maps by the universal property of the ideal completion of a distributive lattice, and by Lemma \ref{isalatticemap9}. Respecting the union of these three relations amounts to the following conditions.
\[
\begin{cases}
\sy{f^+(x^+)}\we \sy{\overline{\sim f^+}(\up x^+)}=0\mb{ whenever }x^+\in L^+\\
\sy{f^+(x^+)}\ve \sy{\overline{\sim f^+}(\up x^+)}=1\mb{ whenever }x^+\in L^+\\
\sy{f^+(a^+)}\we \sy{f^-(a^-)}\leq \sy{f^+(b^+)}\ve \sy{f^-(b^-)}\mb{ when }
 \syw{a}{a}\leq \syv{b}{b}\mb{ in }L,
\end{cases}
\]
as well as the analogues of the first two equalities for generators as well as the analogues of the first two equalities for generators $x^-\in L^-$. The first two equalities hold because $\overline{\sim f^+}(\up x^+)=\sim f^+(x^+)$, and the elements $f^+(x^+)$ and $\sim f^+(x^+)$ are indeed bicomplements of each other in $\ca{M}$, by the assumption that $f$ provides bicomplements to all elements of $L^+\cup L^-$. For the last inequality, we observe that, as $f$ is assumed to be a biframe map, indeed we have that the map $f:L\ra M$ generated by $f^+$ and $f^-$ is such that it respects the inequalities induced by $C_L$.

Now, let us show that the isomorphism between the two biframes is given by the map described in the statement. Indeed, by the fact that our two biframes share the same universal property, with the canonical embedding $\na:\ca{L}\ra \Ad$ corresponding to the syntactic map $\sy{\sy{-}}:\ca{L}\ra \mf{A}^m(\ca(L))/(\mc{Com}^+_L\cup \mc{Com}^-_L\cup \sy{C_L})$, we know that the isomorphism acts as $\na(x^+)\mapsto \sy{\sy{x^+}}$ on positive generators, and similarly on the negative ones. Since it is an isomorphism, it must also respect complements, from which we deduce that this isomorphism is also such that $\del(x^-)\mapsto \sy{\sym{\up a}}$.
\end{proof}

\begin{lemma}\label{finitarylemma91}
If $L^+$ and $L^-$ are frames, and if $C^+$ is a congruence on $L^+$ and $C^-$ a congruence on $L^-$, we have an isomorphism of frames
\[
\cp{(L^+/C^+)}{(L^-/C^-)}\cong(\cpu{L})/\syv{C}{C},
\]
where $\syv{C}{C}=\{(\syp{x}),\syp{y}):(x^+,y^+)\in L^+\}\cup \{(\sym{x},\sym{y}):(x^-,y^-)\in C^-\}$.
\end{lemma}
\begin{proof}
To show that the two frames are isomorphic to each other, we show that they share the same universal property. In particular, we will show that coproduct pairing of the two canonical maps of generators $[-]_{C^+}\oplus [-]_{C^-}:\cpu{L}\ra \cp{(L^+/C^+)}{(L^-/C^-)}$ is such that, whenever there is frame map $f:\cpu{L}\ra M$ which respects the congruence $\syv{C}{C}$, there is a unique frame map $\Tilde{f}:\cp{(L^+/C^+)}{(L^-/C^-)}\ra M$ such that the following triangle commutes. We have used the fact that, by the universal property of the frame coproduct, we may assume that the map $f$ is a coproduct pairing $\cpu{f}$.
\[
\begin{tikzcd}
{} & \cp{(L^+/C^+)}{(L^-/C^-)} 
\arrow{dr}{\Tilde{f}} \\
\cpu{L}
\arrow{ur}{[-]_{C^+}\oplus [-]_{C^-}}
\arrow{rr}{\cpu{f}} 
&& M
\end{tikzcd}
\]
Suppose that $(x^+,y^+)\in C^+$. This implies that $(\syp{x},\syp{y})\in \syp{C}$, and so $f(\syp{x})=f(\syp{y})$, by assumption on the map $f$. Since we have that $f=\cpu{f}$, this also implies that $f^+(x^+)=f^+(y^+)$ whenever $(x^+,y^+)\in C^+$. Then, by definition of frame quotient, there is a frame map $\tilde{f}^+:L^+/C^+\ra M$ such that $f^+(x^+)=\tilde{f}^+([x^+]_{C^+})$ for every $x^+\in L^+$. Consider the coproduct pairing $\cp{\Tilde{f}^+}{\Tilde{f}^-}:\cp{(L^+/C^+)}{(L^-/C^-)}\ra M$ of the map $\tilde{f}^+$ with the analogous one from the negative quotient. This is such that, for every $x^+\in L^+$, we have that $f(x^+)=f^+(x^+)=\Tilde{f}^+([x^+]_{C^+})$. The action of this map on each element is completely determined since we have its action on all generators. Then, it is the unique morphism making the triangle commute.
\end{proof}

\begin{lemma}\label{finitarylemma92}
When we have a biframe $\ca{L}$ and two relations $R$ and $S$ on $L$,
 we have that
 \[
 \ca{L}/(R\cup S)=(\ca{L}/R)/[S]_R,
 \]
 where $[S]_R=\{([x]_R,[y]_R):(x,y)\in S\}$. Furthermore, if $S$ is finitary, so is $[S]_R$.
\end{lemma}
\begin{proof}
To show that the two structures are the same biframe, it suffices to show that they are the same quotient of $\ca{L}$. To show this, we will show that a morphism $f:\ca{L}\ra \ca{M}$ factors through one quotient if and only if it factors through the other one. The set of biframe morphisms in $\bd{BiFrm}(\ca{L},\ca{M})$ which factors through the first quotient is
\[
\{\pa{f}:f:L\ra M\mb{ respects $R$ and $S$}\},
\]
while that of those which factor through the second is
\[
\{\pa{f}:f\mb{ respects $R$},f(x)\leq f(y)\mb{ whenever }([x]_R,[y]_R)\in [S]_R\}.
\]
If a morphism is in the first set, then whenever we have that $([x]_R,[y]_R)\in [S]_R$, by definition of this relation we must have a pair $(x',y')\in S$ such that $(x,x'),(y,y')\in R$. Since the morphism is in the first set, this implies that $f(x)=f(x')\leq f(y')=f(y)$, and so it is in the second set, too. For the converse, suppose that $f$ is in the second set. If we have $(x,y)\in S$, then we also have that $([x]_R,[y]_R)\in [S]_R$, and so, by assumption on $f$, we must also have $f(x)\leq f(y)$, too. Since $f$ respects $S$, it is in the first set. For the second part of the claim, suppose that $S$ is finitary. To show the desired claim, it suffices to show that $[S]_R$ is only constituted of pairs of finitary elements. A typical element of $[S]_R$ is of the form $([x]_R,[y]_R)$ for some $(x,y)\in S$, by definition. The quotient $[-]_R:L\ra L/R$ preserves all the frame operations, and so it maps finitary elements to finitary elements. Then, indeed the pairs in $[S]_R$ are all pairs of finitary elements.  
\end{proof}

\begin{lemma}\label{finitarylemma9}
If we have two frames $L^+$ and $L^-$ and a finitary congruence $C$ on $\cpu{L}$, the biframe $(L^+,L^-,\cpu{L})/C=(L^+/C^+,L^-/C^-,(\cpu{L})/C)$ is finitary.
\end{lemma}
\begin{proof}
To show our claim, we will show that the biframe $\ca{L}/C$ is isomorphic to a finitary biframe. We have that $C=\syp{C}\ve \sym{C}\ve C'$, for some congruence $C'$ on $\cpu{L}$, where $\syp{C}=\{(\syp{x},\syp{y}):(x^+,y^+)\in C^+\}$, and the congruence $\sym{C}$ is defined similarly. By \ref{finitarylemma92}, this biframe is the same as $(\ca{L}/(\syv{C}{C}))/[C']_{(\sy{C^+}\cup \sy{C^-})}$. The congruence $[C']_{(\sy{C^+}\cup \sy{C^-})}$ must be finitary, by the second part of the claim of Lemma \ref{finitarylemma92}. By Lemma \ref{finitarylemma91}, we also have that 
\[
(L^+/C^+,L^-/C^-,\cp{(L^+/C^+)}{(L^-/C^-)})\cong (\ca{L}/(\syv{C}{C})),
\]
and so indeed our biframe is isomorphic to the biframe $(L^+/C^+,L^-/C^-,\cp{(L^+/C^+)}{(L^-/C^-)})$ quotiented by a finitary congruence. 
\end{proof}

We are finally ready to show that we have a well-defined assignment $\mf{A}:\mf{Obj}(\bd{BiFrm_{fin}})\ra \mf{Obj}(\bd{BiFrm_{fin}})$.

\begin{proposition}
The assembly of a finitary biframe is finitary.
\end{proposition}
\begin{proof}
This follows from Proposition \ref{assemblyiso} and Lemma \ref{finitarylemma9}, and from the observation that for any finitary biframe $\ca{L}$ the congruence induced by the relation $C_L\cup \mc{Com}^+_L\cup \mc{Com}^-_L$ is finitary, as these three disjuncts are all relations involving only finitary elements.
\end{proof}

Let us now extend the definition of $\mf{A}$ to morphisms. For every biframe map $f:\ca{L}\ra \ca{M}$, let us denote as $f:L\ra M$ the frame map generated by $f^+$ and $f^-$. By functoriality of the assembly construction on frames, we have a map $\mf{A}(f):\mf{A}(L)\ra \mf{A}(M)$, defined as $\na(x)\mapsto \na(f(x))$ and $\del(x)=\del(f(x))$ on generators. With the next lemma, we show that there is a map $\mf{A}_{\fin}(L)\ra \mf{A}_{\fin}(M)$ defined analogously, and that this map is the main frame component of the required map $\mf{A}(f):\Ad\ra \mf{A}(\ca{M})$.  

\begin{lemma}
When we have a biframe map $f:\ca{L}\ra \ca{M}$, there is a biframe map $\mf{A}(f):\mf{A}(\ca{L})\ra \mf{A}(\ca{M})$ defined as $\na(\syp{x})\mapsto \na(\sy{f^+(x^+)})$ and $\del(\sym{x})\mapsto \del(\sy{f^-(x^-)})$ on positive generators, and similarly on negative ones.
\end{lemma}
\begin{proof}
To show this it suffices to show that the frame map $\mf{A}(f):\mf{A}(L)\ra \mf{A}(M)$ is such that it maps finitary congruences to finitary congruences. A typical finitary congruence of $L$ is of the form $\bve_i \na(\syp{x})\cap \na(\sym{x})\cap \del(\syp{y})\cap \del(\sym{y})$. Since $\mf{A}(f)$ is a frame map, and by its definition, we know that it maps this congruence to $\bve_i \na(\sy{f^+(x^+)})\cap \na(\sy{f^-(x^-)})\cap \del(\sy{f^+(y^+)})\cap \del(\sy{f^-(y^-)})$. By definition, this is a finitary congruence of $M$.
\end{proof}

By definition of the assignment $\mf{A}:\mf{Mor}(\bd{BiFrm_{fin}})\ra \mf{Mor}(\bd{BiFrm_{fin}})$, this respects identities and compositions. We have then proved the following fact. 

\begin{proposition}
The assignment $\mf{A}:\bd{BiFrm_{fin}}\ra \bd{BiFrm_{fin}}$ is a functor.
\end{proposition}

Let us now work towards relating the bispectrum of a finitary biframe with the bispectrum of its assembly.

\begin{lemma}\label{bijectionpoints9}
There is a bijection $\alpha_{\ca{L}}:\bpt(\ca{L})\cong \bpt(\mf{A}(\ca{L}))$ assigning to $f:\ca{L}\ra \bd{2}$ the unique map $\Tilde{f}$ making the following diagram commute.
\[
\begin{tikzcd}
{} 
& \Ad
\arrow{dr}{\tilde{f}} \\
\ca{L}
\arrow{ur}{\na} 
\arrow{rr}{f} 
&& \bd{2}
\end{tikzcd}
\]
\end{lemma}
\begin{proof}
The claim holds by the universal property of the assembly of a biframe: observe that any biframe map $f:\ca{L}\ra \bd{2}$ provides bicomplements to all elements of $L^+\cup L^-$. 
\end{proof}

\begin{lemma}\label{lemmabihomeo9}
For every biframe $\ca{L}$ we have the following.
\begin{itemize}
    \item $\alpha_{\ca{L}}[\va{L}{+}{x}]=\varphi^{+}_{\Ad}(\na(\syp{x}))$,
    \item $\alpha_{\ca{L}}[\va{L}{-}{x}^c]=\varphi^{+}_{\Ad}(\del(\sym{x}))$,
    \item $\alpha_{\ca{L}}[\va{L}{-}{x}]=\varphi^{-}_{\Ad}(\na(\sym{x}))$,
    \item $\alpha_{\ca{L}}[\va{L}{+}{x}^c]=\varphi^{-}_{\Ad}(\del(\syp{x}))$.
\end{itemize}
\end{lemma}
\begin{proof}
Let us prove the first two items.
\begin{itemize}
    \item For $x^+\in L^+$, we have the following chain of equalities. For the equality between the second and the third line, we have used the fact that by Lemma \ref{bijectionpoints9} we know that $\alpha_{\ca{L}}:f\mapsto \tilde{f}$ is a bijection between the points of $\ca{L}$ and those of $\Ad$, and so it is surjective. 
    \begin{align*}
        & \alpha_{\ca{L}}[\va{L}{+}{x}]=\\
        & =\{\tilde{f}\in \bpt(\Ad): f\in \bpt(\ca{L}),f^+(x^+)=1\}=\\
        & =\{\tilde{f}\in \bpt(\Ad):\tilde{f}(\na(\syp{x}))=1\}=\\
        & =\varphi^{+}_{\Ad}(\na(\syp{x})).
    \end{align*}
    \item For $x^-\in L^-$, we have the following chain of equalities. Similarly as above, we have used the fact that $\alpha_{\ca{L}}$ is surjective for the equality between the second and the third line. We have also used the fact that, since $\na(\sym{x})$ and $\del(\sym{x})$ are bicomplements of each other in $\Ad$, we have that $\tilde{f}(\na(\syp{x}))=f^-(x^-)=1$ if and only if $\tilde{f}(\del(\sym{x}))=0$.
     \begin{align*}
        & \alpha_{\ca{L}}[\va{L}{-}{x}^c]=\\
        & =\{\tilde{f}: f\in \bpt(\ca{L}),f^-(x^-)=0\}=\\
        & =\{\tilde{f}\in \bpt(\Ad):\tilde{f}(\del(\syp{x}))=1\}=\\
        & =\varphi^{+}_{\Ad}(\del(\sym{x})).\qedhere
    \end{align*}
\end{itemize}
\end{proof}
\begin{proposition}\label{assemblyskula9}
For every biframe $\ca{L}$ we have that the bijection $\alpha_{\ca{L}}:Sk(\bpt(\ca{L}))\cong \bpt(\mf{A}(\ca{L}))$ is a bihomeomorphism.
\end{proposition}
\begin{proof}
To show that the bijection $\alpha_{\ca{L}}$ is a bihomeomorphism, it suffices to show that the subbasic opens of each topology of $\bpt(\Ad)$ is the forward image of the subbasic opens of the corresponding topology of $Sk(\bpt(\ca{L}))$. Each component of the bispatialization map turns finite meets into finite intersections and arbitrary joins into arbitrary unions. This means that a subbasis of the positive topology of $\bpt(\Ad)$ is given by 
\[
\{\varphi_{\Ad}^{+}(\na(\syp{x})):x^+\in L^+\}\cup \{\varphi_{\Ad}^{+}(\del(\sym{x})):x^-\in L^-\}.
\]
By definition of Skula bispace, and by Lemma \ref{lemmabihomeo9}, these indeed are the forward images under $\alpha_{\ca{L}}$ of the subbasic opens of the positive topology of $Sk(\bpt(\ca{L}))$.
\end{proof}

\begin{theorem}
The following diagram commutes up to natural isomorphism.
\[
\begin{tikzcd}[row sep=large, column sep = large]
\bd{BiFrm_{fin}}^{op} 
\arrow{r}{\mf{bpt}} 
\arrow[swap]{d}{\mf{A}} 
& \bd{BiTop} 
\arrow{d}{Sk} \\
\bd{BiFrm_{fin}}^{op}   
\arrow{r}{\mf{bpt}} & 
\bd{BiTop}
\end{tikzcd}
\]
\end{theorem}
\begin{proof}
Suppose that there is a morphism $f:\ca{L}\ra \ca{M}$ in $\bd{BiFrm_{fin}}$. The naturality square is as follows.
\[
\begin{tikzcd}[row sep=large, column sep = large]
Sk(\bpt(\ca{M}))
\arrow{r}{\alpha_{\ca{M}}} 
\arrow[swap]{d}{Sk(\bpt(f))} 
& \bpt(\mf{A}(\ca{M}))
\arrow{d}{\bpt(\mf{A}(f))} \\
Sk(\bpt(\ca{L}))  
\arrow{r}{\alpha_{\ca{L}}} & 
\bpt(\mf{A}(\ca{L}))
\end{tikzcd}
\]
Since this amounts to an equality of functions, commutativity of this diagram amounts to commutativity of the following square in $\bd{Set}$. We have also used the definition of the functor $\bpt$.
\[
\begin{tikzcd}[row sep=large, column sep = large]
\bpt(\ca{M})
\arrow{r}{\alpha_{\ca{M}}} 
\arrow[swap]{d}{-\circ f} 
& \bpt(\mf{A}(\ca{M}))
\arrow{d}{-\circ\mf{A}(f)} \\
\bpt(\ca{L})) 
\arrow{r}{\alpha_{\ca{L}}} & 
\bpt(\mf{A}(\ca{L}))
\end{tikzcd}
\]
Suppose, then, that $g\in \bpt(\ca{M})$. We need to check that $\alpha_{\ca{M}}(g)\circ \mf{A}(f)=\alpha_{\ca{L}}(g\circ f)$. To check that this is an equality of points of $\Ad$, it suffices to show that these two maps agree on where they map $\na(\syp{x})$, since every bipoint $h:\Ad\ra \bd{2}$ is completely determined to where it maps the closed congruences. For $x^+\in L^+$, we have the following chain of equalities.
\begin{align*}
    & \alpha_{\ca{M}}(g)^+(\mf{A}(f)^+(\na(\syp{x})))=\\
    & =\alpha_{\ca{M}}(g)^+(\na(\sy{f^+(x^+)}))=\\
    & =\tilde{g}^+(\na(\sy{f^+(x^+)}))=\\
    & =g^+(f^+(x^+))=\\
    & =\widetilde{g\circ f}^+(\na(\syp{x})).\qedhere
\end{align*}
\end{proof}

Having now analyzed the definition of the functor $\mf{A}$ and its relation with the bitopological Skula functor $Sk:\bd{BiTop}\ra \bd{BiTop}$, we speak about its connection with the collection of biquotients of a finitary biframe. With the following theorem, we summarize what we already know on this matter. In the next theorem we summarize the importance of the assembly in the theory of finitary biframes.

\begin{theorem}
For any finitary biframe $\ca{L}$, we have the following.
\begin{itemize}
    \item The biframe $\mf{A}(\ca{L})$ has the universal property that it provides bicomplements freely to all elements of $L^+\cup L^-$. 
    \item The main component of $\Ad$ is anti-isomorphic to the coframe $\mf{S}(\ca{L})$ of all biquotients of $\ca{L}$.
    \item There is a bihomeomorphism $\alpha_{\ca{L}}:Sk(\bpt(\ca{L}))\cong \bpt(\Ad)$
    \item The patch-closed sets of $\bpt(\Ad)$, under the bijection $|\bpt(\ca{L})|\cong|\bpt(\mf{A}(\ca{L}))|$, coincide with the underlying sets of the bisober bisubspaces of $\bpt(\ca{L})$.
\end{itemize}
\end{theorem}
\begin{proof}
The first part of the claim follows from the fact that the biframe assembly has the universal property proven in \cite{schauerte92}. The second part of the claim follows from Corollary \ref{assemblyiso} and Lemma \ref{AfinLisSL}. The third and fourth parts of the claim follow from Proposition \ref{assemblyskula9} and Corollary \ref{generalbisobersub}.
\end{proof}

\section{The relation between bisubspaces and biquotients}

In this section we explore the relation between the biquotients of a finitary biframe and the bisubspaces of its bispectrum. In the course of this analysis, we will encounter a bitopological version of the $T_D$ axiom, and we will argue that this is a natural generalization of the classical $T_D$ axiom in a point-set fashion; afterwards we will see that its role in the finitary bitopological duality is analogous to the role of the $T_D$ axiom in the classical frame duality. Let us start, then by working towards describing bispatialization as a kind of interior operator on $\mf{S}(\ca{L})$, for a finitary biframe $\ca{L}$. Bispatialization, as we know, is a map $\ca{L}\epi \mc{fin}(\mf{b}\Om(\bpt(\ca{L})))$, obtain by composing the classical biframe spatialization with the coreflector $\mc{fin}$. Recall also that for a biframe $\ca{L}$ we have that the lattice $\mf{S}(\ca{L})$ is a coframe, anti-isomorphic to the finitary assembly $\mf{A}_{\fin}(L)$ of the main component of $\ca{L}$. 

\begin{lemma}\label{halfconverse3}
If $L$ is a complete lattice and $I\se L$ is a collection closed under arbitrary joins, the operator $\iota:L\ra L$ defined as $\iota:x\mapsto \bve \{y\in I:y\leq x\}$ is an interior operator.
\end{lemma}
\begin{proof}
For monotonicity, we notice that if $x\leq z$ we must have $\{y\in I:y\leq x\}\se \{y\in I:y\leq z\}$ and so $\iota(x)\leq \iota(z)$. It is clear that $\iota(x)\leq x$ for every $x\in L$. Since the collection $I$ is closed under arbitrary joins, we have that $\iota(x)\in I$ for every $x\in L$. Then, $\iota(x)$ is the maximum of the collection $\{y\in I:y\leq \iota(x)\}$. By definition, then, $\iota(\iota(x))=\iota(x)$. Then, $\iota$ is monotonic, decreasing, and idempotent. 
\end{proof}

\begin{lemma}\label{interiorsubframe}
If $L$ is a frame, and $F\se L$ is a subframe of $L$, the interior operator $\iota_F:a\mapsto \bve \{f\in F:f\leq a\}$ preserves finite meets.
\end{lemma}
\begin{proof}
The map $\iota_F$ defined in the claim is an interior operator by Lemma \ref{halfconverse3}. Consider $a_1,...,a_n\in L$. We now show that $\{f\in F:f\leq a_1\we ...\we a_n\}=\{f_1\we ...\we f_n:f1,...,f_n\in F, f_m\leq a_m\}$. The inclusion of the first set into the second is obvious. For the converse, we notice that if $f_m\leq a_m$ then also $f_1\we ...\we f_n\leq a_1\we ...\we a_n$. We then have that $\iota_F(a_1\we ...\we a_n)=\bve \{f\in F:f\leq a_1\we ...\we a_n\}=\bve \{f_1\we ...\we f_n:f1,...,f_n\in F,f_m\leq a_m\}=\bve \{f\in F:f\leq a_1\}\we ...\we \bve \{f\in F:f\leq a_n\}=\iota_F(a_1)\we ...\we \iota_F(a_n)$. We have used the frame distributivity law for the second equality of this chain.
\end{proof}

In \cite{picadopultr2011frames}, the following is shown. There, the result is phrased in terms of sublocales; below we simply rephrase it in terms of assemblies and congruences.
\begin{proposition}\label{spaconucleus}
Spatialization $\spa:\mf{A}(L)\ra \mf{A}(L)$ is a nucleus: it is a closure operator which preserves finite meets.
\end{proposition}

Let us exploit this result to show that something analogous holds for the frame of biquotients of a finitary biframe.

\begin{proposition}\label{conucleusagain}
For a finitary biframe $\ca{L}$, the operator $\bisp:\mf{S}(\ca{L})\ra \mf{S}(\ca{L})$ is an interior operator which preserves finite joins. 
\end{proposition}
\begin{proof}
Let $\ca{L}$ be a finitary biframe. For a congruence $C$ on the main component $L$, let us denote as $\spa (C)$ the congruence corresponding to the frame theoretical spatialization quotient $q_{\varphi_{L}}:L\epi \spa (L)$. To show our claim, we will show that $\bisp=\mc{fin}\circ \spa:\mf{A}_{\fin}(L)\ra \mf{A}_{\fin}(L)$ is a closure operator which preserves finite meets. We recall that since $\mf{A}_{\fin}(L)$ is a subframe of $\mf{A}(L)$ the finite meets in $\mf{A}_{\fin}(L)$ are computed as set theoretical intersections. To see that $\mc{fin}\circ \spa$ a closure operator on the finitary congruences, we consider an arbitrary finitary congruence $C$ on $L$ and we observe that for a congruence $C$ we have $C\se \spa(C)$; we then recall that $\mc{fin}:\mf{A}(L)\ra \mf{A}(L)$ is monotone, and so $\mc{fin}(C)=C\se \mc{fin}(\spa (C))$. For idempotency, we observe that bispatial finitary biframes are isomorphic to their bispatialization. We recall (see Proposition \ref{spaconucleus}) that for a frame $M$ and congruences $C_m$ on it, we have that $\spa(C_1\cap ...\cap C_n)=\spa (C_1)\cap ...\cap (C_n)$. This means that for congruences $C_m$ on our main component $L$, we have that $\spa:\mf{A}(L)\ra \mf{A}(L)$ preserves finite meets. Since the finitary congruences of $L$ form a subframe of $\mf{A}_{\fin}(L)$, by Lemma \ref{interiorsubframe} above the interior $\mc{fin}:\mf{A}(L)\ra \mf{A}(L)$ preserves finite meets. So, when we have a collection $C_m$ of finitary congruences of $L$, we have $\bisp(C_1\cap ...\cap C_n)=\mc{fin}(\spa (C_1\cap ...\cap C_n))=\bisp(C_1)\cap ...\cap \bisp (C_n)$.
\end{proof}

\begin{corollary}
For any finitary biframe $\ca{L}$ the ordered collection $\bisp[\mf{S}(\ca{L})]$ is a coframe, and the map $\bisp:\mf{S}(\ca{L})\ra \bisp[\mf{S}(\ca{L})]$ is a coframe surjection.
\end{corollary}
\begin{proof}
By Proposition \ref{conucleusagain}, we have that $\bisp:\mf{S}(\ca{L})\ra \mf{S}(\ca{L})$ is a nucleus. Then, we have that the fixpoints of this operators are a sublocale of $\mf{S}(\ca{L})$, in particular they are a frame. Furthermore, since they are a sublocale we have that $\bisp:\mf{S}(\ca{L})^{op}\ra \bisp[\mf{S}(\ca{L})]^{op}$ is a frame surjection.
\end{proof}

In the same way, we may see bisobrification as a closure operator on the ordered collection of all bisubspaces of a bispace. Let $\ca{L}$ be a finitary biframe, and let us as $\ca{P}(\bpt(\ca{L}))$ the powerset of its bispectrum. We call $\bisob:\ca{P}(\bpt(\ca{L}))\ra \ca{P}(\bpt(\ca{L}))$ the map 
\[
Y\mapsto \bca \{Z\se \bpt(\ca{L}):Z \mb{ is bisober, }Y\se Z\}.
\]
Since we know that bisober bisubspaces are closed under arbitrary intersections (as they coincide with those whose underlying set is patch-closed sets of $Sk(\bpt(\ca{L}))$), the bisubspace $\bisob(Y)$ is always bisober for any bisubspace $Y$. It is also clear that the assignment $\bisob:Y\mapsto \bisob (Y)$ is increasing and idempotent. So, it is a closure operator. The inclusion $Y\se \bisob(Y)$ is in fact the bisobrification map, but we omit the proof of this fact as it is not crucial for our investigation. Let us state as a proposition the result that we do need.
\begin{proposition}
For a finitary biframe $\ca{L}$, the map $\bisob:\ca{P}(\bpt(\ca{L}))\ra \ca{P}(\bpt(\ca{L}))$ is a closure operator. Its fixpoints are a subcoframe of $\ca{P}(\bpt(\ca{L}))$.
\end{proposition}
\begin{proof}
We have argued above why the assignment is a closure operator. The fixpoints of this operators are clearly the bisober bisubspaces of $\bpt(\ca{L})$. By Theorem \ref{skulabisobers}, these form a subcoframe of $\ca{P}(\bpt(\ca{L}))$, as their underlying sets coincide with the closed sets of some topology on $\bpt(\ca{L})$.
\end{proof}

It is also clear that the fixpoints of the operator $\bisob:\ca{P}(\bpt(\ca{L}))\ra \ca{P}(\bpt(\ca{L}))$ are exactly the bisober bisubspaces of $\bpt(\ca{L})$. let us exploit the finitary bitopological duality to show that the coframe of bisober bisubspaces of a bispectrum $\bpt(\ca{L})$ is isomorphic to the coframe of bispatial biquotients of $\ca{L}$. In order to connects biquotients and bisubspaces, we wish to see the bispectrum functor $\bpt$ as connecting the lattice $\mf{S}(\ca{L})$ of biquotients of a finitary biframe $\ca{L}$ with the lattice $\ca{P}(\bpt(\ca{L}))$ of all bisubspaces of its bispectrum. We abuse notation slightly, and define a map $\bpt:\mf{S}(\ca{L})\ra \ca{P}(\bpt(\ca{L}))$ defined as
\begin{align*}
    & \bpt:\mf{S}(\ca{L})\ra \ca{P}(\bpt(\ca{L}))\\
    & \ca{L}/R\mapsto \bpt(q_R)[\bpt(\Lq{R})].
\end{align*}
In other words, this operator maps each biquotient $\Lq{R}$ to the concrete bisubspace of $\bpt(\ca{L})$ determined by the dualization $\bpt(q_R):\bpt(\Lq{R})\inclu \bpt(\ca{L})$ of the quotient map $q_R:\ca{L}\ra \Lq{R}$. 

\begin{proposition}
The assignment $\bpt:\mf{bisp}[\mf{S}(\ca{L})]\ra \mf{bisob}[\ca{P}(\bpt(\ca{L}))]$ is an isomorphism of coframes.
\end{proposition}
\begin{proof}
We begin by pointing out that, since the functor $\bpt$ establishes a dual equivalence of categories between bispatial finitary biframes and bisober bispaces, when seen as a map between these two categories it is full, faithful, and essentially surjective. Let us show that the map defined in the claim is injective. We need to show that if the subspace inclusions $\bpt(q_R):\bpt(\Lq{R})\inclu \bpt(\ca{L})$ and $\bpt(q_S):\bpt(\Lq{S})\inclu \bpt(\ca{L})$ correspond to the same bisubspace of $\bpt(\ca{L})$ for two biquotient maps $q_R:\ca{L}\epi \Lq{R}$ and $q_S:\ca{L}\epi \Lq{S}$ to bispatial finitary biframes, we must also have $q_R=q_S$. Suppose, then, that the antecedent holds. Then, there must be an isomorphism $h:\bpt(\Lq{R})\cong \bpt(\Lq{S})$ such that the following triangle commutes.
\[
\begin{tikzcd}
    \bpt(\Lq{R}) 
    \arrow{rr}{h(\cong)} 
    \arrow[dr,swap,"\bpt(q_R)",hookrightarrow] 
    & & \bpt(\Lq{S}) 
    \arrow[dl,"\bpt(q_S)",hookrightarrow] \\
    & \bpt(\ca{L})  
\end{tikzcd}
\]
Suppose also that $q_{\phi}:\ca{L}\ra \bisp (\ca{L})$ is the bispatialization biquotient map. Let $\overline{q_R}:\bisp(\ca{L})\ra \Lq{R}$ and $\overline{q_S}:\bisp(\ca{L})\ra \Lq{S}$ be the unique maps such that $\overline{q_R}\circ q_{\phi}=q_R$ and $\overline{q_S}\circ q_{\phi}=q_S$, which exist as we have assumed the two biquotients to be bispatial. We start from the equality $\bpt(q_R)\circ h=\bpt(q_S)$, which is true by commutativity of the triangle above, and deduce that $\bpt(\overline{q_R}\circ q_{\phi})\circ h=\bpt(\overline{q_S}\circ q_{\phi})$, that is $\bpt(\overline{q_R})\circ \bpt(q_{\phi})\circ h=\bpt(\overline{q_S})\circ \bpt(q_{\phi})$. By Lemma \ref{unitbecomesiso3}, the map $\bpt(q_{\phi})$ is an isomorphism. Since isomorphisms are always epimorphisms, we deduce $\bpt(\overline{q_R})=\bpt(\overline{q_S})$. Since $\overline{q_R}$ and $\overline{q_S}$ are maps between bispatial finitary biframes, we may use faithfulness of $\bpt$ to deduce $\overline{q_R}=\overline{q_S}$. Then, $\overline{q_R}\circ q_{\phi}=\overline{q_S}\circ q_{\phi}$, that is, $q_R=q_S$.  

On the other hand, by Lemma \ref{bisobersarequotients} every bisubspace inclusion $i:\bpt(\ca{M})\inclu \bpt(\ca{L})$ is, up to bihomeomorphism, a bisubspace inclusion of the form $\bpt(q_R):\bpt(\Lq{R})\ra \bpt(\ca{L})$, for some relation $R$ on $L$ and for $\Lq{R}$ bispatial. Since the bispectrum functor is contravariant, the assignment from bispatial biquotients to bisober bisubspaces that we have defined is order preserving and reflecting.  
\end{proof}

We can begin the investigation of the relation between bisubspaces and biquotients by looking at the following diagram. The following is a commutative diagram in the category of coframes. 

  \begin{center}  
    \begin{tikzcd}[row sep=large, column sep = large]
        \bisp [\mf{S}(\ca{L})]  
        \arrow{r}{\bpt(\cong)}   
        & \bisob[\ca{P}(\bpt(\ca{L}))] 
        \arrow[hookrightarrow]{d}  \\
     \mf{S}(\ca{L}) 
     \arrow{r}{\bpt} 
     \arrow[twoheadrightarrow]{u}{\bisp} 
     & \ca{P}(\bpt(\ca{L}))  
    \end{tikzcd}
 \end{center}

We wish to explore what are those finitary biframes such that the left vertical arrow of the diagram above is an isomorphism of coframes. The question of characterizing the finitary biframes such that the right vertical arrow is an isomorphism (the ``totally spatial" finitary biframes), and also that of characterizing the finitary biframes where both vertical arrows are isomorphisms, for now remain open.

\subsection{The bi-$T_D$ axiom}\label{bitd}

We begin the analysis of the relation between bisubspaces and biquotients by introducing a bitopological version of the $T_D$ axiom for spaces, which in the duality of finitary biframes will play a similar role to that played by its monotopological counterpart in the classical frame duality. We define a bispace to be \tc{bi}-$T_D$ if every singleton is of the form $U^+\cap U^-\cap V^+{}^c\cap V^-{}^c$ for positive opens $U^+,V^+$ and negative opens $U^-,V^-$. 

\begin{lemma}\label{inclusioninduces}
For a bispace $X$ and a bisubspace $Y\se X$, we have that the dualization of the bisubspace inclusion $i:Y\se X$ is the map $-\cap Y:\bomf(X)\ra \bomf(Y)$. 
\end{lemma}
\begin{proof}
Suppose that $X$ is a bispace, and consider a bisubspace inclusion $i:Y\se X$. By definition of the functor $\bomf$, we have that the map $\bomf(i):\bomf(X)\ra \bomf(Y)$ is defined for a positive open $U^+\se X$ as $i^{-1}(U^+)=\{y\in Y:i(y)\in U^+\}=U^+\cap Y$. 
\end{proof}

Let us see the first proposition in which we list equivalent characterizations of bi-$T_D$ spaces, which re-trace classical monotopological ones of the $T_D$ axiom.  

\begin{proposition}\label{bitdcharacterize}
For a bispace $X$ the following are equivalent.
\begin{enumerate}
    \item The bispace $X$ is bi-$T_D$.
    \item The patch of $Sk(X)$ is discrete.
    \item Whenever we have distinct bisubspaces $Y,Z\se X$ they induce a different biquotient on $\bomf(X)$.
    \item Whenever we have a point $x\in X$ the dualization of the inclusion $i:X{\sm}\{x\}\se X$ is not an isomorphism of finitary biframes.
\end{enumerate}
\end{proposition}
\begin{proof}
Item (1) implies item (2) by definition of the Skula bitopology of a bispace. Conversely, if (2) holds, then every singleton $\{x\}$ must be of the form $\bcu_i U^+_i\cap U^-_i\cap V_i^+{}^c\cap V_i^-{}^c$, but since it is a singleton if it is of this form it must be $U^+_i\cap U^-_i\cap V_i^+{}^c\cap V_i^-{}^c$ for some $i\in I$, and so (1) holds. Now, suppose that (1) holds. Let $Y,Z\se X$ be bisubspaces such that $Y\nsubseteq Z$; in particular, let $x\in Y{\sm}Z$. Let $\{x\}=U^+\cap U^-\cap V^+{}^c\cap V^-{}^c$. We have that the biquotient induced by $Y$ is that induced by the biframe surjection $-\cap Y:\Om_{\fin}(X)\epi \Om_{\fin}(Y)$, and similarly for $Z$, by Lemma \ref{inclusioninduces}. Then, these two cannot be the same, as we have that $U^+\cap U^-\cap Y\nsubseteq (V^+\cup V^-)\cap Y$ but $U^+\cap U^-\cap Z\se (V^+\cup V^-)\cap Z$. We see that item (3) implies item (4) by setting $Y=X$ and $Z=X{\sm}\{x\}$. Suppose that item (4) holds, and let $x\in X$. By assumption, we must have that the dualization $\Om(i):\bomf(X)\ra \bomf(X{\sm}\{x\})$ of the inclusion $i:X{\sm}\{x\}\se X$ induces a finitary congruence on $\bomf(X)$ other than the identity. Since finitary congruences are induced by relations of the form $\{(\sy{U^+_i\cap U^-_i},\sy{V^+_i\cup V^-_i}):i\in I\}$ on $\Om^+(X)\oplus \Om^-(X)$, we must have that $U^+\cap U^-\nsubseteq V^+\cup V^-$ and $(U^+\cap U^-)\cap (X{\sm}\{x\})\subseteq (V^+\cup V^-)\cap (X{\sm}\{x\})$ for some positive opens $U^+,V^+$ and negative opens $U^-,V^-$. This means that we must have $\{x\}=U^+\cap U^-\cap V^+{}^c\cap V^-{}^c$.
\end{proof}
We observe that the bi-$T_D$ axiom can be seen as a natural bitopological version of the classical $T_D$ axiom for spaces, consider for instance these two arguments. 
\begin{itemize}
    \item A topological space $X$ is $T_1$ if and only if every singleton is closed. It is discrete if and only if every singleton is open. The $T_D$ axiom weakens both these in that every singleton is the intersection of an open and a closed set. Classically, pairwise bispaces are exactly those such that every singleton is a closed set of the form $U^+{}^c\cap U^-{}^c$. Discrete bispaces are those such that their patch-is discrete, that is, those such that every singleton is of the form $V^+\cap V^-$. We obtain the bi-$T_D$ property by weakening both these in the same way. A bi-$T_D$ space is one where every singleton is of the form $U^+{}^c\cap U^-{}^c\cap V^+\cap V^-$.
    \item We have already seen that the bitopological version $Sk:\bd{BiTop}\ra \bd{BiTop}$ of the Skula construction is assigns to a bispace $X$ the bispace whose underlying set is $|X|$, whose positive opens are the topology generated by the positive opens and the closed sets of $X$, and whose negative opens are the topology generated by the negative opens of $X$ and its positive closed sets. It is one of the main characterizations of the classical $T_D$ axiom that a space $X$ is $T_D$ if and only if the Skula space of $X$ is discrete. A bispace $X$ is bi-$T_D$ if and only if $Sk(X)$ is patch-discrete. If we accept that $Sk:\bd{BiTop}\ra \bd{BiTop}$ is a good bitopological version of the Skula construction, then we ought to also accept that the bi-$T_D$ property is a good bitopological version of the $T_D$ one.
\end{itemize}

The following lemma shows that the bi-$T_D$ axiom interacts with bisobriety as $T_D$ does with sobriety, suggesting that if the bi-$T_D$ property is a natural bitopological version of the $T_D$ one, then bisobriety is a natural bitopological version of sobriety.

\begin{proposition}\label{allbisubarebisob}
For a d-frame $\ca{L}$ the following are equivalent.
\begin{enumerate}
    \item All bisubspaces of $\bpt(\ca{L})$ are bisober.
    \item The collection of bisober subspaces of $\bpt(\ca{L})$ is closed under arbitrary unions.
     \item The collection of bisober subspaces of $\bpt(\ca{L})$ is closed under complementation.
     \item The bisubspaces of the form $\bpt(\ca{L}){\sm}\{f\}$ are bisober.
    \item The space $\bpt(\ca{L})$ is bi-$T_D$.
\end{enumerate}
\end{proposition}
\begin{proof}
It is immediate that (1) implies (2), (3), and (4). Let us show the three reverse directions. Item (2) implies item (1) because singletons are all bisober, and every bisubspace is a union of singletons. To see that (3) implies (1), we observe that an arbitrary bisubspace is of the form $S=(\bcu_i \{f_i\})^c=\bca_i \{f_i\}^c$. Because the collection of bisober subspaces is closed under arbitrary intersections, and because singletons are all d-sober, if each complement $\{f_i\}^c$ is bisober then $S$ is bisober, too. The same argument also shows that (4) implies (1). Finally, let us show that (4) implies (5). Suppose that every bisubspace of the form $\bpt(\ca{L})$ is d-sober. Then, for every point $f\in \bpt(\ca{L})$ we cannot possibly have that the inclusion $i:\bpt(\ca{L})\se \bpt(\ca{L})$ when dualized becomes a d-frame isomorphism, or we would contradict that $\bomf$ is part of an equivalence of categories between bisober bispaces and bispatial finitary biframes. This means that the space $\bpt(\ca{L})$ is bi-$T_D$, by comparing items (1) and (3) of Proposition \ref{bitdcharacterize}. Then, (4) implies (5). Finally, let us show that (5) implies (1). Suppose that there is a bisubspace with inclusion $i_Y:Y\se \bpt(\ca{L})$ which is not bisober. Dualizing this we and precomposing with the spatialization map we get a biframe surjection $ \bomf(i_Y)\circ\varphi_{\ca{L}}:\ca{L}\ra \bomf(Y)$. Let $q_R:\ca{L}\ra\ca{L}/R$ be the corresponding biquotient. We know that $\bpt(\Lq{R})$ is a subspace of $\ca{L}$. It cannot be $Y$, as this is assumed to not be bisober. By its construction it also induces on $\bomf(\bpt(\ca{L}))$ the same quotient as $Y$. By comparing items (1) and (2) of Proposition \ref{bitdcharacterize}, we see that this means that $\bpt(\ca{L})$ is not bi-$T_D$. 
\end{proof}

In the space $\bpt(\ca{L})$, we are going to have biquotients with many points, and these are the biquotients which are generated by finitary relations only containing one pair. A typical finitary relation is the union of relations of the form $\{(a^+\we a^-,b^+\ve b^+)\}$. We will now see that a bispace $\bpt(\ca{L})$ is bi-$T_D$ if and only if every such biquotient takes away from $\bpt(\ca{L})$ only one point.

\begin{lemma}
For a finitary biframe $\ca{L}$ and a point $f\in \bpt(\ca{L})$ we have that $\{f\}=\va{L}{+}{a}\cap \va{L}{-}{a}\cap \va{L}{+}{b}^c\cap \va{L}{-}{b}^c$ if and only if $\{f\}^c=\bpt(\ca{L}/\{(a^+\we a^-,b^+\ve b^-)\})$. 
\end{lemma}
\begin{proof}
This follows from the fact that the relation $\{(a^+\we a^-,b^+\ve b^-)\}$ induces on $\ca{L}$ the congruence $\na(a^+\we a^-)\cap \del(b^+\ve b^-)$, and by item (5) of Proposition \ref{manyfacts9}.
\end{proof}

We obtain the following theorem.

\begin{theorem}
For a biframe $\ca{L}$, the following are equivalent.
\begin{enumerate}
    \item The space $\bpt(\ca{L})$ is bi-$T_D$.
    \item The bispace $\bpt(\Ad)\cong Sk(\bpt(\ca{L}))$ is patch-discrete.
    \item The coframe $\bisp[\mf{S}(\ca{L})]$ of bispatial biquotients is a Boolean algebra.
    \item All the bisubspaces of $\bpt(\ca{L})$ are bisober.
    \item The map $\bpt:\bisp[\mf{S}(\ca{L})]\ra\ca{P}(\bpt(\ca{L}))$ is an isomorphism.
\end{enumerate}
\end{theorem}
\begin{proof}
The equivalence between (1) and (4) is part of Proposition \ref{allbisubarebisob}. The equivalence between (2) and (4) is part of Proposition \ref{bitdcharacterize}. For the equivalence between (4) and (5) we look at the diagram before Subsection \ref{bitd}. Since the top horizontal arrow is an isomorphism, we have that $\bpt:\bisp[\mf{S}(\ca{L})]\ra \ca{P}(\bpt(\ca{L}))$ is an isomorphism if and only if the right vertical arrow is an isomorphism, that is, if and only if all bisubspaces of $\bpt(\ca{L})$ are bisober. Finally, it is clear that (5) implies (3). For the converse, suppose that $\bisp[\mf{S}(\ca{L})]$ is a Boolean algebra. Then, $\bisob[\ca{P}(\bpt(\ca{L}))]$, too, is a Boolean algebra, as it is isomorphic to it. In the coframe $\bisob[\ca{P}(\bpt(\ca{L}))]$ finite meets and joins are finite unions and finite intersections, respectively, since this coframe is a subcoframe of $\ca{P}(\bpt(\ca{L}))$. Then, $\bisob[\ca{P}(\bpt(\ca{L}))]$ being a Boolean algebra implies that every bisober bisubspace of $\bpt(\ca{L})$ has a set-theoretical complement. In other words, the collection of bisober bisubspaces is closed under complementation. By Proposition \ref{allbisubarebisob}, this implies that all bisubspaces of $\bpt(\ca{L})$ are bisober.
\end{proof}

\bibliographystyle{elsarticle-harv}
\bibliography{bib}

\end{document}